\documentclass[12pt]{amsart}
\usepackage{amssymb}
\usepackage{amsmath,amscd}
\usepackage{amsfonts,latexsym,verbatim,amscd,mathrsfs,color,array}
\usepackage[all,cmtip]{xy}
\usepackage{mathrsfs}
\usepackage{multirow}
\usepackage{graphicx}
\usepackage{amsfonts, amsthm}
\usepackage{bbm}
\usepackage[hmargin=1in,vmargin=1in]{geometry}
\usepackage{mathdots}
\usepackage{stmaryrd}
\usepackage{MnSymbol}
\usepackage{bbm}

\allowdisplaybreaks

\def\XXint#1#2#3{{\setbox0=\hbox{$#1{#2#3}{\int}$ }
		\vcenter{\hbox{$#2#3$ }}\kern-.6\wd0}}

\renewcommand{\Re}{\operatorname{Re}}
\renewcommand{\Im}{\operatorname{Im}}
\DeclareMathSymbol{\intprod}{\mathbin}{MnSyC}{'270}
\newcommand{\LB}{\left[}
\newcommand{\RB}{\right]}
\newcommand{\LA}{\langle}
\newcommand{\RA}{\rangle}

\newcommand{\C}{{\mathbb C}}

\newcommand{\R}{{\mathbb R}}

\newcommand{\rB}{{\mathrm{B} }}
\newcommand{\rH}{{\mathrm{H} }}
\newcommand{\rSU}{{\mathrm{SU} }}
\newcommand{\rSp}{{\mathrm{Sp} }}
\newcommand{\rU}{{\mathrm U}}
\newcommand{\rSO}{{\mathrm{SO} }}
\renewcommand{\rB}{{\mathrm B}}
\newcommand{\rG}{{\mathrm G}}
\newcommand{\rK}{{\mathrm K}}

\newcommand{\rSpin}{{\mathrm{Spin} }}

\newcommand{\del}{{\partial}}

\newcommand{\gothg}{{\mathfrak g}}
\newcommand{\gothso}{\mathfrak {so }}
\newcommand{\gothk}{{\mathfrak k}}
\newcommand{\gothh}{{\mathfrak h}}
\newcommand{\gothu}{{\mathfrak u}}
\newcommand{\gothb}{{\mathfrak b}}

\newcommand{\pd}{{\partial}}

\newcommand{\inj}{{\mathrm{inj} }}
\newtheorem{thm}{Theorem}[section]
\newtheorem{lemma}[thm]{Lemma}
\newtheorem*{lemma*}{Lemma}
\newtheorem{prop}[thm]{Proposition}

\newtheorem{cor}[thm]{Corollary}
\newtheorem{conj}[thm]{Conjecture}
\newtheorem*{conj*}{Conjecture}

\newtheoremstyle{others}
{3pt}
{2pt}
{}
{}
{\bf}
{.}
{.5em}
{}

\theoremstyle{others}
\newtheorem{rmk}[thm]{Remark}
\newtheorem*{rmk*}{Remark}
\newtheorem{defn}[thm]{Definition}
\newtheorem{example}[thm]{Example}

\numberwithin{equation}{section}

\begin{document}

\title{Yang-Mills flow on special-holonomy manifolds}
\author{Gon\c{c}alo Oliveira and Alex Waldron}
\address{Universidade Federal Fluminense IME--GMA, Niter\'oi, Brazil}
\email{galato97@gmail.com}
\address{Michigan State University, East Lansing, MI}
\email{awaldron@msu.edu}

\begin{abstract}
This paper 
develops Yang-Mills flow on Riemannian manifolds with special holonomy. By analogy with the second-named author's thesis, we find that a supremum bound on a certain curvature component is sufficient to rule out finite-time singularities. 
Assuming such a bound, we prove that the infinite-time bubbling set is calibrated by the defining $(n-4)$-form.

\end{abstract}

\maketitle

\tableofcontents

\thispagestyle{empty}

\section{Introduction}



\subsection{Background}

Let $(M^n,g)$ be an oriented Riemannian manifold and $E$ a vector bundle over $M,$ with metric $\LA \cdot, \cdot \RA.$ The metric-compatible connections on $E$ form an affine space, $\mathcal{A}_E,$ modeled on the space of 
1-forms valued in the adjoint bundle. 
Letting $F_A$ denote the curvature of $A \in \mathcal{A}_E,$ we may define the \emph{Yang-Mills energy}
$$\mathcal{E}(A)= \frac{1}{2} \int_M |F_A|^2 \, dV.$$
This functional originates in physics and has been studied extensively by both physicists and mathematicians. 

The negative gradient of the Yang-Mills energy defines a vector field on $\mathcal{A}_E,$ 
whose integral curves correspond to paths of connections $A(t)$ solving
\begin{equation}\tag{YM}\label{YM}
\frac{\partial A}{\partial t} = - D_A^* F_A.
\end{equation}
This equation generates a semi-parabolic flow on $\mathcal{A}_E$ referred to as \emph{Yang-Mills flow}, whose fixed points 
are known as \emph{Yang-Mills connections}. 
Running (YM) towards an infinite-time limit is a natural strategy for finding these critical points.

The global study of the Yang-Mills functional was initiated by the foundational paper of Atiyah and Bott \cite{Atiyah1983}. 
G. Daskalopoulos \cite{Daskalopoulos1992} was able to recover Atiyah-Bott's stratified picture of $\mathcal{A}_E$ over a Riemann surface via direct analysis of (YM), and R\aa de \cite{Rade1992} later showed that the flow always exists and converges smoothly in subcritical dimensions $(n < 4).$ Still, the set of critical points of $\mathcal{E}$ on a 3-manifold has not been thoroughly described. 

Since the work of Taubes \cite{taubesselfdual} and Uhlenbeck \cite{uhlenbeckremov, uhlenbecklp}, which enabled the development of Donaldson Theory, the critical dimension ($n=4$) has been of particular interest in both the elliptic and the parabolic setting. Struwe \cite{Struwe1993} developed a robust existence theory and blowup criterion for (YM) in dimension four. Schlatter, Struwe, and Tahvildar-Zadeh \cite{Schlatter1998} were able to prove long-time existence for solutions of (YM) by assuming rotational symmetry on $\R^4,$ leading them to conjecture that finite-time singularities do not occur in general. This conjecture was settled by the second-named author \cite{Waldron2016}. S.-J. Oh and Tataru have also recently achieved a significant breakthrough for the critical Yang-Mills equations in the hyperbolic setting, where the flow (YM) plays an auxiliary role 
\cite{ohtataru}. 

In supercritical dimensions ($n > 4$), the known results concerning (YM) are sharply divergent. 
On one hand, assuming that $M$ is compact K\"ahler and $E$ is a holomorphic vector bundle, with $A(0)$ a compatible connection, we have the remarkable result of Donaldson \cite{Donaldson1985, donaldsoninfinitedeterminants} 
that (YM) can be extended smoothly for all time. 
Infinite-time behavior is then dictated by the underlying algebraic structure \cite{Simpson1988,Daskalopoulos2004,Daskalopoulos2007,Sibley2015,Sibley2015_2}. 

On the other hand, leaving aside the holomorphic setting, 
we know that finite-time singularities exist in abundance \cite{Naito1994}. These are expected to be ``rapidly forming'' singularities driven by simple parabolic rescaling---such \emph{Type-I} singularities, and the shrinking solitons on which they are modeled, have been investigated by several authors \cite{Gastel2002, Weinkove2004,
 Kelleher2016b}. 
Moreover, as evidenced by the recent work of Donninger and Sch\"orkhuber \cite{DonningerSchorkhuber2016}, singularities of this type are relatively stable. It is likely that for any bundle with nonabelian structure group over a manifold of dimension $n > 4,$ large open subsets of $\mathcal{A}_E$ necessarily develop Type-I singularities when evolved by (YM).

The purpose of this paper is to explore the space between these two regimes. In the spirit of Donaldson and Thomas \cite{Donaldson1998}, we shall continue to work \emph{a priori} with the full space of connections $\mathcal{A}_E,$ while restricting the holonomy of the base manifold. 
The present goals are to determine the minimal blowup criteria for (YM) under the assumption that $M$ has special holonomy, 
and to refine the classification of singularities. 
The future goal will be to establish properties of (YM) comparable to those which it enjoys in the 4-dimensional and K\"ahler cases, for appropriate classes of initial data over exceptional-holonomy manifolds. 

\subsection{Statement of results}


Recall that a closed $k$-form $\Psi$ on $(M,g)$ is said to be a $k$-\emph{calibration} if, for any $x \in M$ and orthonormal set $\{ e_1 , \ldots, e_k \} \subset T_x M,$ there holds
\begin{equation}\label{calibrated}
\Psi (e_1, \ldots, e_k) \leq 1.
\end{equation}
Assuming that $M$ carries an $(n-4)$-calibration, according to Corrigan, Devchand, Fairlie, and Nuyts \cite{corrigandevchand}, and Tian \cite{tiancalibrated} (see also Reyes-Carri\'on \cite{reyescarrion}), the key notion of four-dimensional Yang-Mills theory can be generalized to higher dimensions as follows. An \emph{instanton} (or \emph{$\Psi$-instanton}) is defined to be a connection $A$ solving
\begin{equation}\label{instanton}
F_A + * \left( \Psi \wedge F_A \right) = 0.
\end{equation}
Solutions of (\ref{instanton}) are Yang-Mills connections, indeed minimizers of $\mathcal{E}$ under appropriate circumstances (see Remark \ref{rmk:minimizers}).

By contrast with dimension four, the symmetry between (\ref{instanton}) and a complementary ``anti-instanton'' equation is broken, the latter typically being overdetermined. Moreover, the 2-forms often split into more than two irreducible components on special-holonomy manifolds (see \S \ref{sec:splitting}-\ref{sec:berger}). For simplicity, we state our main results  here only in the case of exceptional holonomy, and refer the reader to \S \ref{sec:berger} and \S  \ref{sec:blowupcriteria} for the remaining cases on Berger's list.

On a manifold with holonomy $\rG_2,$ the curvature of any connection splits as
\begin{equation}\label{eq:Decomposition_Introduction}
F_A = F_A^7 + F_A^{14}.
\end{equation}
Here $F_A^{14}$ is the curvature component which lies in the bundle associated to the Lie algebra $\gothg_2,$ 
and $F_A^7$ is the orthogonal complement; a similar splitting is valid in the $\rSpin(7)$ case. Letting $\Psi = \phi$ (resp. $\Psi = \Theta$) be the defining form of the $\rG_2$ (resp. $\rSpin(7)$)-structure, the instanton equation (\ref{instanton}) takes the form
$$F_A^7 = 0.$$
Instantons on compact manifolds with exceptional holonomy are typically very difficult to construct. $\rSpin(7)$-instantons were first studied in depth by Lewis \cite{Lewis1998}, and nontrivial examples (with structure group $\rSU(2)$) were later constructed by Tanaka \cite{Tanaka2012}. The first examples of $\rG_2$-instantons (with structure group $\rSO(3)$) on compact manifolds with full $\rG_2$ holonomy were constructed by Walpuski \cite{Walpuski2013}.

The splitting (\ref{eq:Decomposition_Introduction}) may also be exploited in a parabolic context. Prior to studying the more general question in dimension four, the second-named author proved in \cite{instantons} that long-time existence holds for (YM) provided that $F_A^+$ or $F_A^-$ is small in $L^2,$ \textit{i.e.}, the energy is nearly minimal. Our main theorem is a partial generalization of this result to the special- holonomy setting. 

\begin{thm}[Cf. Theorem \ref{thm:mainblowupcrit}]\label{thm:Main_1_Introduction}
Let $M$ be a compact Riemannian manifold with holonomy contained in $\rG_2$ or $\rSpin(7),$ and assume that $A(t)$ is a smooth solution of (YM) over $M \times \LB 0, T \right),$ with $T < \infty.$ If
\begin{equation}\label{f7bounded}
\sup_{\stackrel{t < T}{x\in M} } | F_A^7 (x,t) | < \infty
\end{equation}
then the flow extends smoothly beyond time $T.$ An identical result holds with $F_A^{14}$ (in the $\rG_2$ case) or $F_A^{21}$ (in the $\rSpin(7)$ case) replacing $F_A^7.$
\end{thm}

The main technical result behind Theorem \ref{thm:Main_1_Introduction} is an extended version of Hamilton's monotonicity formula \cite{Hamilton1993}, Theorem \ref{thm:mainthm}, which is operative in the special-holonomy situation. Theorem \ref{thm:mainthm} leads to an enhancement, Theorem \ref{thm:modifiedepsilon}, of the well-known epsilon-regularity theorem for (YM) in higher dimensions; the latter was proven by Chen and Shen \cite{chenshen}, and goes back to Struwe \cite{struwehm} in the harmonic-map-flow context. 
The technical version of Theorem \ref{thm:Main_1_Introduction} (Theorem \ref{thm:mainblowupcrit}) follows straightforwardly from Theorem \ref{thm:modifiedepsilon}. 
At the same time, within the general setup below, 
we obtain a modern proof of long-time existence in the compact K\"ahler case, Corollary \ref{cor:Long_Time_Kahler}, and a new existence result for Yang-Mills flow over compact quaternion-K\"ahler manifolds, Corollary \ref{cor:quatkahler}. 

Next, recall that a closed oriented $(n-4)$-dimensional submanifold $N \subset M$ is said to be \emph{calibrated} by $\Psi$ if, for any $x \in M,$ equality holds in (\ref{calibrated}) for an orthonormal basis 
of $T_xN.$ 
Any such submanifold minimizes volume in its homology class \cite{Harvey1982}. This notion extends to the more general class of \emph{rectifiable sets}, \emph{i.e.}, subsets of $M$ which are (up to measure zero) composed of countable unions of Lipschitz $(n-4)$-dimensional submanifolds, with locally finite $(n-4)$-Hausdorff measure. An $(n-4)$-rectifiable set is calibrated if (\ref{calibrated}) holds for $\mathcal{H}^{n-4}$-almost-every point, where indeed its tangent spaces are well-defined. 

According to the theorem of Tian \cite{tiancalibrated}, energy-minimizing Yang-Mills connections (instantons) and $\Psi$-calibrated sets are related via the bubbling process. This relationship has been a focus of intense study, in view of the proposal by Donaldson and Segal \cite{Donaldson2009} to construct gauge-theoretic invariants by counting instantons together with the calibrated submanifolds along which they may concentrate energy. 
Codimension-four calibrated submanifolds (or rectifiable sets) are known as \emph{associative} in the case of $\rG_2$ holonomy, and \emph{Cayley} in the case of $\rSpin(7)$ holonomy.

In the parabolic context, Hong and Tian \cite{hongtian} have proven that the infinite-time singular set $\Sigma$ of a global solution of (YM) is rectifiable, and in the K\"ahler case, supports a holomorphic current. 
Our second main result extends Hong-Tian's characterization 
from K\"ahler manifolds to the broader class of special-holonomy manifolds, assuming a bound on the relevant curvature component. 
This gives a conditional generalization of the relationship between gauge theory and calibrated geometry to the parabolic setting. 

\begin{thm}[Cf. Theorem \ref{thm:calibrated}]\label{thm:introthm2}
	Let $A(t)$ be a smooth solution of (YM) on $\LB 0, T \right)$ over a $\rG_2$ (resp. $\rSpin(7)$)-holonomy manifold, with $T$ is maximal, and assume (\ref{f7bounded}). Then $T = \infty,$ and
the singular set $\Sigma$ at infinite time 
is an $(n-4)$-rectifiable, associative (resp. Cayley) subset. 
Furthermore, for $\mathcal{H}^{n-4}$-almost-every $x \in \Sigma$ there is a blowup sequence which converges to an anti-self-dual connection on $(T_x \Sigma)^{\perp}$. 
\end{thm}
\noindent The main technical ingredients of this result form the subject of a companion paper \cite{waldronuhlenbeck}, leaving the proof in \S \ref{sec:calibratedness} very short.

Lastly, in \S \ref{sec:reductions}, we examine the consequences of Theorem \ref{thm:Main_1_Introduction} for a variety of dimensional and holonomy reductions of (YM). We obtain blowup criteria for several new parabolic systems related to the Vafa-Witten, Calabi-Yau-monopole, and $\rG_2$-monopole equations.

\subsection{Acknowledgements} 
Gon\c{c}alo Oliveira would like to thank Casey Kelleher for several conversations during the initial stage of this project. Alex Waldron would like to thank Thomas Walpuski and Karsten Gimre for suggestions and comments on the manuscript. He also thanks the Simons Collaboration on Special Holonomy for support during the academic year 2017-18.


\section{$N(\rG)$-structures and calibrations}\label{sec:splitting}






Let $V$ be an oriented Euclidean vector space of dimension $n.$ Fix a linear $(n-4)$-calibration $\Psi$ on $V,$ \emph{i.e.}, an alternating $(n-4)$-tensor which satisfies (\ref{calibrated}).

Let $\rG$ be a connected simple Lie group acting effectively on $V,$ and suppose that $\Psi$ is preserved by the normalizer $N(\rG) \subset \rSO(V).$ Let $\rK \subset \rSO(V)$ be a Lie subroup which contains $N(\rG):$ 
$$\rG \subset N(\rG) \subset \rK \subset \rSO(V).$$
We denote the respective Lie algebras by
\begin{equation}\label{groupsplitting}
\mathfrak{g} \subset N( \mathfrak{g}) \subset \mathfrak{k} \subset \Lambda^2 V.
\end{equation}
The $(n-4)$-form $\Psi$ defines a self-adjoint, traceless operator
\begin{equation}\label{staroperator}
\ast \left( \cdot \wedge \Psi \right) : \Lambda^2V \to \Lambda^2V.
\end{equation}
We shall assume that this operator preserves $\gothk = \mathrm{Lie}(\rK),$ 
and let
\begin{equation*}
\lambda_\alpha, \quad  \alpha = 0, 1, \ldots, \alpha_{max}
\end{equation*}
be its eigenvalues on $\gothk$ (taken without multiplicity). Writing $\omega^\alpha$ for the $\lambda_\alpha$-component of $\omega \in \gothk,$ we have
\begin{equation}\label{eigenvalues}
\ast(\omega \wedge \Psi) = \lambda_\alpha \omega^\alpha.
\end{equation}
Assume further that $\lambda_0 = -1,$ and
\begin{equation}\label{minusoneassumption}
\gothg = \ker \left( \mathbf{1} + \ast \left( \cdot \wedge \Psi \right) \right).
\end{equation}
In this paper, we shall always attach both the linear form $\Psi$ and the choice of $K,$ satisfying the above assumptions, to the group $\rG \subset \rSO(V).$

Following Reyes-Carri\'on \cite{reyescarrion},
we now consider a manifold $M$ of dimension $n \geq 4$ equipped with an $N(\rG)$-structure. 
Since the adjoint action of $N(\rG)$ preserves $\gothg$ and $\gothk,$ we may form associated subbundles of $\Lambda^2 = \Lambda^2 T^*M$ according to (\ref{groupsplitting}).
The $N(\rG)$-structure also determines a Riemannian metric $g$ on $M,$ as well as a differential $(n-4)$-form $\Psi.$ Globally, we obtain an orthogonal decomposition
\begin{equation}\label{eq:Splitting}
\Lambda^2 \cong \mathfrak{g} \oplus \left( \gothk \cap \gothg^{\perp} \right) \oplus \mathfrak{k}^{\perp}
\end{equation}
which is preserved by the operator (\ref{staroperator}). 
It is natural to relabel the components
$$ \Lambda^2_- = \mathfrak{g}, \qquad \Lambda^2_+= \gothk \cap \mathfrak{g}^{\perp}, \qquad \Lambda^2_{\perp} = \gothk^{\perp}$$
so that (\ref{eq:Splitting}) becomes
\begin{equation}\label{eq:Splitting_2}
\Lambda^2 = \Lambda^2_- \oplus \Lambda^2_+ \oplus \Lambda^2_{\perp}.
\end{equation}

Finally, let $E \to M$ be a vector bundle, and denote by $\gothso(E) \subset \mbox{End}(E)$ the Lie-algebra bundle associated to the structure group of $E,$ which we take to be $\rSO(\mathrm{rk}(E))$ for simplicity. 
The $\gothso(E)$-valued $2$-forms split orthogonally as
\begin{equation}\nonumber
\Lambda^2 \otimes \gothso(E) = (\Lambda^2_- \otimes \gothso(E)) \oplus (\Lambda^2_+ \otimes \gothso(E)) \oplus (\Lambda^2_{\perp} \otimes \gothso(E)).
\end{equation}
The curvature $F = F_A$ of any connection $A$ on $E$ decomposes correspondingly:
\begin{equation*}
F  = F^- + F^+ + F^{\perp}.
\end{equation*}
\begin{defn}\label{def:Compatible_Connection}
	We say that a connection 
is $\gothk$-\emph{compatible} if its curvature takes values in $\mathfrak{k} \otimes \gothso(E),$ \textit{i.e.} 
\begin{equation}\label{fperpzero}
F^{\perp} = 0.
\end{equation}
We then have
\begin{equation}\label{falphadecomp}
F = F^- + F^+ = \sum_{\alpha = 0}^{\alpha_{max}} F^\alpha.
\end{equation}
We shall write $\mathcal{A}^\gothk_E$ for the space of $\gothk$-compatible connections on $E.$
\end{defn}

\begin{rmk} The next section provides several examples of this compatibility condition. In the K\"ahler case, we will have $\mathcal{A}^\gothk_E = \mathcal{A}_E^{1,1},$ agreeing with the standard notion of a compatible connection on a holomorphic bundle.
\end{rmk}

\begin{lemma}\label{lemma:dstarf}
Suppose that $d\Psi = 0,$ i.e., $\Psi$ is a calibration. 
Fix a nonzero eigenvalue $\lambda_\beta$ of the operator (\ref{staroperator}), and let
\begin{equation}\label{kappahat}
\kappa_\alpha = \frac{\lambda_\beta - \lambda_\alpha}{\lambda_\beta}.
\end{equation}
Then, for the curvature $F = F_A$ of a $\gothk$-compatible connection $A \in \mathcal{A}^\gothk_E$, 
we have
\begin{equation}\tag{$i$}
D^* F = \kappa_\alpha D^* F^\alpha.
\end{equation}
If $M$ is compact, then
\begin{equation}\tag{$ii$}
\mathcal{E}(A)  = \frac{1}{2 \lambda_\beta} \int_M \LA F \wedge F \RA\wedge \Psi + \frac{\kappa_\alpha}{2} \Vert F^\alpha \Vert^2_{L^2}. 
\end{equation}
where the first term is a constant depending only on $E.$ Here $\alpha$ is summed over in both equations ($i$-$ii$).
\end{lemma}
\begin{proof} Fix $\beta$ throughout the proof. Applying the adjoint $D^*$ to (\ref{eigenvalues}), we have
\begin{equation*}\label{dstarf1}
\begin{split}
\lambda_\alpha D^* F^\alpha & = D^* \ast(F \wedge \Psi)  \\
& = - \ast D \ast^2 \left( F \wedge \Psi \right) \\
& = \pm \ast \left( D F \wedge \Psi + F \wedge d \Psi \right) \\
& = 0.
\end{split}
\end{equation*}
Here we have used the Bianchi identity and the assumption that $\Psi$ is closed. Then
$$D^*F^\beta = -\sum_{\alpha \neq \beta} \frac{\lambda_\alpha}{\lambda_\beta} D^*F^\alpha$$
and
\begin{equation*}
\begin{split}
D^*F  & = D^*F^\beta + \sum_{\alpha \neq \beta} D^*F^\alpha \\
& = \left( - \frac{\lambda_\alpha}{\lambda_\beta} + 1\right) D^*F^\alpha \\
& = \kappa_\alpha D^*F^\alpha
\end{split}
\end{equation*}
which is ($i$).

The proof of ($ii$) follows similarly from the identity
$\langle F \wedge F \rangle \wedge \Psi = \lambda_\alpha  |F^\alpha|^2 \, dV.$ 
\end{proof}

\begin{rmk}\label{rmk:minimizers} Note from (\ref{kappahat}) and ($ii$) that a connection with $F = F^\beta$ is a minimizer of the Yang-Mills energy if $\lambda_\beta$ is either the most negative or most positive eigenvalue of (\ref{staroperator}). Cf. Stern's notion of a ``conservative decomposition'' in \S 3 of \cite{Stern2010}. 
\end{rmk}




\section{Berger's list}\label{sec:berger} Recall that an $N(\rG)$-structure is said to be \emph{torsion-free} if the holonomy of the Levi-Civita connection $\nabla_g$ is contained in $N(\rG).$ 
In this case, since the induced differential form $\Psi$ is parallel (in particular a calibration), $\nabla_g$ preserves the splitting (\ref{eq:Splitting_2}). We shall consider torsion-free $N(\rG)$-structures exclusively.

This section applies the above setup to each item on Berger's list, consisting of the possible restricted holonomy groups of a Riemannian manifold which is neither a locally symmetric space nor a product. The following lemma allows for the condition (\ref{minusoneassumption}) to be easily verified. 

\begin{lemma}\label{lemma:eigen} Let $\Psi$ be a linear $(n - 4)$-calibration on an oriented Euclidean vector space $V,$ and assume that $\rG \subset \rSO(V)$ is a connected simple Lie group which preserves $\Psi.$
If $\rG$ contains an $\rSU(2)$-subgroup which fixes both a calibrated plane $U \subset V$ and the space of self-dual 2-forms $\Lambda^2_+(U^{\perp}),$ then 
\begin{equation*}
		\gothg \subset	\ker  \left( \mathbf{1} + \ast( \cdot \wedge \Psi ) \right).
\end{equation*}
\end{lemma}
\begin{proof} By assumption, $\rG$ contains a subgroup $\rB$ which acts trivially on $U,$ and acts as $\rSU(2)_-$ on $U^\perp \cong \R^4,$ with
$$\gothb \cong \Lambda^2_-(U^\perp).$$
We may decompose $\Psi$ as
\begin{equation*}
\Psi = \mbox{Vol}_U + \sum_{i = 1}^4\Psi_i, \quad \Psi_i \in \Lambda^{n -4-i} U \otimes \Lambda^i U^{\perp}.
\end{equation*}
It follows from (\ref{calibrated}) that $\Psi_1$ must vanish. Further write
\begin{equation*}
\Psi_2 = \Psi_+ + \Psi_-
\end{equation*}
with $\Psi_{\pm} \in \Lambda^2 U \otimes \Lambda^2_{\pm} U^{\perp},$ and
\begin{equation*}
\Psi_- = \sum \alpha_k \otimes \beta_k
\end{equation*}
with $\{ \alpha_k \} \subset \Lambda^{n-6} U$ a linearly independent set and $\beta_k \in \Lambda^2_- U^\perp.$ Since $\rB$ acts trivially on $U$ and fixes $\Psi_-,$ each element $\beta_k$ must be fixed by $\rB.$ But this implies that $\beta_k = 0$ for all $k,$ and we conclude that $\Psi_- = 0.$ 

For any element $\beta \in \gothb = \Lambda^2_- U^\perp,$ we now have
\begin{equation*}
\begin{split}
* \left( \beta \wedge \Psi \right) & = * \left( \beta \wedge \left(\mbox{Vol}_U + \Psi_+ + \Psi_3 + \Psi_4 \right) \right) \\
& = * \left( \beta \wedge \mbox{Vol}_U \right) \\
& = -\beta.
\end{split}
\end{equation*}
Hence $\gothb$ is contained in the $-1$-eigenspace of (\ref{staroperator}). Since $\rG$ commutes with $\ast( \cdot \wedge \Psi ),$ and $\gothg$ forms an irreducible module for the adjoint action, the same must be true for all of $\gothg.$
\end{proof}

\subsection{Four-manifolds} When $n=4,$ we may take
$$\rG=\rSU(2)_-, \qquad \rK = N(\rG) = \rSO(4)$$
and $\Psi=1.$
Then (\ref{eq:Splitting_2}) is the standard decomposition of the $2$-forms into anti-self-dual and self-dual parts, and $A \in \mathcal{A}_E^\gothk = \mathcal{A}_E$ is a general metric-compatible connection.

\subsection{$\rU(k)$-manifolds}\label{ss:Kahler}

Let $(M^{n},\omega, I)$ be a K\"ahler manifold, with $n=2k$. Let $$\rG=\rSU(k), \qquad \rK = N(\rG) = \rU(k)$$
and
$$\Psi= \frac{\omega^{k-2}}{(k-2)!}.$$
The decomposition (\ref{eq:Splitting}) reads
$$\Lambda^2 = \underbrace{\mathfrak{su}(k) \oplus \mathfrak{u}(1)}_{\mathfrak{u}(k)} \oplus \mathfrak{u}(k)^{\perp}$$
which corresponds to
\begin{equation*}
\Lambda^2 = \Re \Lambda^{1,1}_0 \oplus \LA \omega \RA \oplus \Re \left( \Lambda^{2,0} \oplus \Lambda^{0,2} \right).
\end{equation*}
The three summands have dimensions
$$(k+1)(k-1), \quad 1, \quad k(k-1)$$
respectively, with eigenvalues
$$-1, \quad k-1, \quad 1$$
for the operator (\ref{staroperator}). These eigenvalues can be verified using Lemma \ref{lemma:eigen}, tracelessness of (\ref{staroperator}), and irreducibility of the summands as $\rSU(k)$-modules.

Letting $E \to M$ be a hermitian vector bundle, a connection $A \in \mathcal{A}^{\gothu(k)}_E = \mathcal{A}_E^{1,1}$ 
is compatible with the corresponding holomorphic structure on $E.$
According to Lemma \ref{lemma:dstarf}$i,$ we have
\begin{equation*}
D^*F = k D^*F^{\mathfrak{u}(1)} = \frac{k}{k-1} D^*F^{\mathfrak{su}(k)}.
\end{equation*}

\subsection{$\rSp(k)\rSp(1)$-manifolds}\label{ss:quatkahler}

Let $(M^{n}, \Omega)$ be a quaternion-K\"ahler manifold, with $n=4k.$ In this case, we pick
$$\rG=\rSp(k), \qquad \rK = N(\rG) =\rSp(k) \rSp(1).$$
The fundamental 4-form $\Omega$ (also known as the Kraines form \cite{kraines}) is closed and may be written locally as
$$\Omega=  \omega_1 \wedge \omega_1 + \omega_2 \wedge \omega_2 + \omega_3 \wedge \omega_3$$
where $\omega_1, \omega_2, \omega_3$ are a triple of (local) nondegenerate $2$-forms associated with $3$ orthogonal almost complex structures satisfying the quaternion relations. Then, the $4l$-forms given by $\tfrac{\Omega^l}{(2l+1)!}$ are calibrations (see \cite{bryantharvey}, Theorem 6.3), and we let
$$\Psi = \frac{\Omega^{k-1}}{(2k-1)!}.$$ 

The decomposition (\ref{eq:Splitting}) takes the form
\begin{equation}\label{qksplitting}
\begin{split}
& \mathfrak{sp}(k) \oplus \mathfrak{sp}(1) \oplus (\mathfrak{sp}(k) \oplus \mathfrak{sp}(1))^{\perp} \\
\hat{=} & \,\,\,\, \Lambda^2_{P} \,\,\, \oplus \,\,\, \Lambda^2_{\Omega} \,\,\,\,\, \oplus \,\,\, \Lambda^{2}_{o}.
\end{split}
\end{equation}
The three summands have dimensions
$$2k^2 + k, \quad 3, \quad 4k^2 -3k -3$$
respectively, with eigenvalues
$$-1, \quad \frac{2k + 1}{3}, \quad \frac{1}{3}$$
for the operator (\ref{staroperator}), as computed by Galicki and Poon \cite{Galicki1991}.
The components of (\ref{qksplitting}) also have the following local description (see \cite{capriasalamon}, Proposition 1). Choosing a local basis $I_1,I_2,I_3$ of almost-complex structures satisfying the quaternion relations, we have
$$\Lambda^2_P= \bigcap_{i = 1}^3 \Lambda^{1,1}_{I_i}, \qquad \Lambda^2_{\Omega} \oplus \Lambda^{2}_{o} = \sum_{i = 1}^3\Re \Lambda^{2,0}_{I_i}.$$ 

With the above choice of $K,$ a $\gothk$-compatible connection $A \in \mathcal{A}_E^{\mathfrak{sp}(k) \mathfrak{sp}(1)}$ will be called \emph{pseudo-holomorphic}. 
From Lemma \ref{lemma:dstarf}$i$, we have
\begin{equation*}
D^*F = \frac{2k+4}{3} D^*F^{\Omega} = \frac{2k+4}{2k+1} D^*F^{P}.
\end{equation*}

\begin{rmk} As a reference for gauge theory on quaternion-K\"ahler manifolds, the reader may see the recent paper by Devchand, Pontecorvo, and Spiro \cite{devchandpontecorvo}.
\end{rmk}

\subsection{$\rG_2$-manifolds}\label{ss:g2splitting}

Let $(M^7,\phi)$ be a $\rG_2$-manifold, with $\phi$ the defining (co)-closed $3$-form. In this case, we take
$$\rG = N(\rG)=\rG_2, \qquad \rK=\rSO(7)$$
and $\Psi= \phi.$ The decomposition (\ref{eq:Decomposition_Introduction}) is
$$\mathfrak{g}_2 \oplus \mathfrak{g}_2^{\perp} \cong \Lambda^2_{14} \oplus \Lambda^2_{7}$$
with eigenvalues $-1$ and $2,$ respectively, for (\ref{staroperator}).
For an arbitrary metric-compatible connection $A \in \mathcal{A}_E^\gothk = \mathcal{A}_E,$ Lemma \ref{lemma:dstarf}$i$ implies
\begin{equation*}
D^*F = 3 D^*F^7 = \frac{3}{2} D^*F^{14}.
\end{equation*}

\subsection{$\rSpin(7)$-manifolds}\label{ss:spin7splitting}

Let $(M^8,\Theta)$ be a $\rSpin(7)$-manifold, with $\Theta$ the closed Cayley $4$-form. In this case, we take
$$\rG = N(\rG) =\rSpin(7), \qquad \rK = \rSO(8)$$
and $\Psi= \Theta$. Then, as in the previous case, the decomposition (\ref{eq:Splitting}) is
$$\mathfrak{spin}(7) \oplus \mathfrak{spin}(7)^{\perp} \cong \Lambda^2_{21}M \oplus \Lambda^2_{7}M, $$
with eigenvalues $-1$ and $3,$ respectively, for (\ref{staroperator}).
For a connection $A \in \mathcal{A}_E^\gothk = \mathcal{A}_E,$ Lemma \ref{lemma:dstarf}$i$ implies
\begin{equation*}
D^*F = 4 D^*F^7 = \frac{4}{3} D^*F^{21}.
\end{equation*}

\subsection{$\rSU(k)$-manifolds}
Calabi-Yau manifolds 
of arbitrary dimension may be treated as general K\"ahler manifolds, while Calabi-Yau 3 and 4-folds may also be treated within the $\rSpin(7)$ framework (see \S \ref{sec:reductions}).  

\subsection{$\rSp(k)$-manifolds} Hyperk\"ahler manifolds may be treated either as K\"ahler manifolds, where compatible connections are holomorphic with respect to a fixed integrable complex structure,
or as quaternion-Kahler manifolds, where compatible connections are pseudo-holomorphic (see \S \ref{ss:quatkahler} above).

\section{Curvature evolution under (YM)}\label{sec:Weitzenbock_Evolution}


We recall several basic properties of solutions of the flow (YM) defined above (see \textit{e.g.} \cite{instantons}, \S 2, for derivations).

Given a smooth initial connection on a bundle over a compact manifold $M$ (and  indeed more generally), short-time existence of a smooth solution $A(t)$ of (YM) follows from a version of the De Turck trick due to Donaldson (see \cite{donkron}, \S 6). As with other geometric flows, this solution will exist smoothly on a maximal time interval $M \times \LB 0, T \right),$ with $T \leq \infty.$ A supremum bound on the full curvature $F_{A(t)}$ is sufficient to extend the flow smoothly at finite time (see \textit{e.g.} Lemma 2.4 of \cite{waldronuhlenbeck}); conversely, if the maximal existence time $T$ is finite, we must have
\begin{equation*}\label{supremumblowup}
\limsup_{t \nearrow T} \| F_{A(t)} \|_{L^\infty(M)} = \infty.
\end{equation*}

For a solution of (YM), the curvature $F = F_{A(t)}$ evolves through 
\begin{equation}\label{pointwisecurvevol}
\frac{\partial F}{\partial t} = -\Delta_A F
\end{equation}
where
\begin{equation}\label{hodgelaplace}
\Delta_A = D^* D + D D^*
\end{equation}
is the Hodge Laplacian associated to the fixed metric $g$ and the evolving connection $A = A(t).$
Integrating in space and time against $F,$ and applying the Bianchi identity $DF = 0,$ we obtain the \emph{(global) energy identity}
\begin{equation}\label{eq:Energy_Identity}
\mathcal{E}(A(t))  =  \mathcal{E}(A(0)) - 2\int_0^t \!\!\! \int_M |D^*F(x,s)|^2 dV(x) \ ds
\end{equation}
for any $0 \leq t < T.$

Analytic results concerning (YM) often depend on a combination of localized or specialized versions of (\ref{eq:Energy_Identity}), and Bochner/Weitzenb\"ock formulae applied to (\ref{pointwisecurvevol}). We shall take an elementary approach to the latter; for more sophisticated treatments, see \cite{freeduhl}, Appendix, or \cite{Semmelman2010}. 

\subsection{Weitzenb\"ock formulae}\label{ss:Weitzenbock}

We shall use the geometer's convention (\ref{hodgelaplace}) for the Laplace operator.

Given an arbitrary smooth connection $A \in \mathcal{A}_E,$ by coupling with Levi-Civita, one obtains a connection $\nabla$ on each bundle of $\mathfrak{so}(E)$-valued differential forms. 
For $\omega \in \Omega^k(\mathfrak{so}(E)),$ the standard Weitzenbock formula (see e.g. \cite{instantons}, p. 7) reads
\begin{equation}\label{weitz}
\begin{split}
\left( \Delta_A \omega \right)_{i_1 \cdots i_k} = \left( \nabla^* \nabla \omega \right)_{i_1 \cdots i_k} - \LB F_{i_1}{}^j , \omega_{j i_2 \cdots i_k} \RB - \cdots - \LB F_{i_k}{}^j , \omega_{i_1 \cdots i_{k-1} j} \RB + \left( R \# \omega \right)_{i_1 \cdots i_k}
\end{split}
\end{equation} 
where the term $R \# \omega$ is described as follows. Working in normal coordinates, for $\alpha$ an $(m + 2)$-tensor and $\omega$ a $k$-tensor, define the binary operation
\begin{equation*}\label{circleriemann}
\begin{split}
\left( \alpha \circ \omega \right){}_{i_1 \cdots i_m \ell_1 \cdots \ell_k} 
& = \alpha_{i_1 \cdots i_m \ell \ell_1} \omega_{\ell \ell_2 \cdots \ell_k} + \alpha_{i_1 \cdots i_m \ell\ell_2} \omega_{\ell_1 \ell \ell_3 \cdots \ell_k} + \cdots + \alpha_{i_1 \cdots i_m\ell\ell_k} \omega_{\ell_1 \ell_2 \cdots \ell}.
\end{split}
\end{equation*}
Then the last term in (\ref{weitz}) is defined by
\begin{equation*}\label{rhash}
\begin{split}
\left( R \# \omega \right)_{i_1 \cdots i_k} = & (R \circ \omega)_{i_1jj i_2 \cdots i_n} +  (R \circ \omega)_{i_2 j i_1 j \cdots i_n} + \cdots + (R \circ \omega)_{i_n j i_1 \cdots j}.
\end{split}
\end{equation*}
For an $\mathfrak{so}(E)$-valued 1-form $\alpha,$ (\ref{weitz}) therefore reads
\begin{equation}\label{weitz1}
\begin{split}
\left( \Delta_A \alpha\right)_{i} & = -\nabla_k \nabla_k \alpha_{i} - \LB F_{ik} , \alpha_{k} \RB + \mathrm{Ric}_{ik} \alpha_k.
\end{split}
\end{equation}
For a 2-form $\omega \in \Omega^2 \left( \mathfrak{so}(E) \right),$ (\ref{weitz}) reads
\begin{equation}\label{weitz2}
\begin{split}
\left( \Delta_A \omega \right)_{ij} & = - \nabla_k \nabla_k \omega_{ij} - \LB F_{ik} , \omega_{k j} \RB + \LB F_{jk} , \omega_{k i} \RB \\
& \qquad \qquad \qquad + R_{ik\ell k} \omega_{\ell j} +  R_{ik\ell j} \omega_{k \ell} - R_{jk\ell k} \omega_{\ell i} -  R_{jk\ell i} \omega_{k \ell} \\
& = - \nabla_k \nabla_k \omega_{ij} - \LB F_{ik} , \omega_{k j} \RB + \LB F_{jk} , \omega_{k i} \RB \\
& \qquad \qquad \qquad + R_{i k j \ell} \omega_{\ell k} - R_{ik k \ell} \omega_{\ell j} - \left( R_{j k i \ell } \omega_{ \ell k} - R_{j k k \ell} \omega_{\ell i} \right).
\end{split}
\end{equation}

In order to rewrite (\ref{weitz1}-\ref{weitz2}) intrinsically, we let a 2-form $\gamma$ act on a 1-form $\alpha$ by 
$$\alpha_j \mapsto \left(\gamma \cdot \alpha\right)_j = \gamma_{jk} \alpha_k.$$
This action extends to $\mathfrak{so}(E)$-valued forms 
by the rule
$$\alpha_j \mapsto \LB \gamma \cdot \alpha \RB_j := \LB \gamma_{jk}, \alpha_k \RB.$$
We will use the bold bracket $\llbracket, \rrbracket$ to denote the commutator on 2-forms
\begin{equation}\label{doublebracket}
\llbracket \omega, \eta \rrbracket_{ij} = \omega_{ik} \eta_{kj} - \omega_{jk} \eta_{ki}
\end{equation}
induced by the metric. A real-valued 2-form acts on $\mathfrak{so}(E)$-valued forms by (\ref{doublebracket}), while $\Omega^2(\mathfrak{so}(E))$ acts on itself by
\begin{equation}\label{sonboldbracket}
\omega \mapsto \llbracket \gamma, \omega \rrbracket_{ij} := \LB \gamma_{ik}, \omega_{kj} \RB - \LB \gamma_{jk}, \omega_{ki} \RB.
\end{equation}
Notice that in the case of two $\gothso(E)$-valued forms (by contrast with real ones), we have
\begin{equation}\label{commutatativity}
\llbracket \gamma, \omega \rrbracket = \llbracket \omega, \gamma \rrbracket.
\end{equation}
Using these conventions, (\ref{weitz1}) may be rewritten
\begin{equation}\label{weitz1brackets}
\begin{split}
\Delta_A \alpha & = \nabla^*\nabla \alpha - \LB F \cdot \alpha \RB + \mathrm{Ric} \cdot \alpha.
\end{split}
\end{equation}

On a manifold with holonomy $\rH,$ the Ambrose-Singer theorem states that the Riemann curvature tensor $R$ takes values in $\mathrm{Sym}^2 \gothh,$ \textit{i.e.}, there exist real-valued functions $\rho_a$ and 2-forms $\varepsilon_a \in \gothh$ such that
$$R = \sum_a \rho_a \varepsilon_a \otimes \varepsilon_a$$
or, in components
$$R_{ijk\ell} = \sum_a \rho_a \left( \varepsilon_a \right)_{ij} \left( \varepsilon_a \right)_{k \ell}.$$
We may therefore rewrite (\ref{weitz2}) as
\begin{equation}\label{weitz2brackets}
\Delta_A \omega = \nabla^* \nabla \omega  - \llbracket F , \omega \rrbracket - \sum_{a} \rho_a \llbracket \varepsilon_a , \llbracket \varepsilon_a , \omega \rrbracket \rrbracket.
\end{equation}

Finally, applying (\ref{weitz2brackets}) to the evolution (\ref{pointwisecurvevol}) of the curvature $F = F_{A(t)}$ under (YM), we obtain
\begin{equation}\label{ymcurvatureweitz}
\left( \frac{\pd}{\pd t} + \nabla^*\nabla \right) F = \llbracket F , F \rrbracket + \sum_{a} \rho_a \llbracket \varepsilon_a , \llbracket \varepsilon_a , F \rrbracket \rrbracket.
\end{equation}
Taking an inner product with $F,$ we obtain the basic differential inequality
\begin{samepage}
\begin{equation}\label{ymcurvatureweitzwithnorms}
\left( \frac{\pd}{\pd t} + \Delta \right) |F|^2 \leq - 2 | \nabla F |^2 + C |F |^3 + C_M |F|^2.
\end{equation}

\subsection*{Note} 
The norm $| \cdot |$ generally denotes the standard pointwise norm on $\gothso(E)$-valued differential forms, e.g. $|F|^2 = \frac{1}{2} \LA F_{ij}, F_{ij} \RA,$ although in (\ref{ymcurvatureweitz}) we have also used the norm on $\Lambda^1 \otimes \Lambda^2$ to write $ | \nabla F |^2 = \frac{1}{2} \LA \nabla_i F_{jk}, \nabla_i F_{jk} \RA.$
\end{samepage}

\subsection{Compatibility}\label{ss:Compatibility} 
We shall now use the above Weitzenb\"ock formulae to prove that (YM) preserves the space of $\gothk$-compatible connections, 
per Definition \ref{def:Compatible_Connection}. For K\"ahler manifolds, this fact is typically proved via the K\"ahler identities. 

\begin{prop}\label{prop:Orthogonal_Curvature_Vanishing}
	Assume that $M$ is a compact manifold equipped with a torsion-free $N(\rG)$-structure as in \S \ref{sec:splitting}. 
If $A(t)$ is a smooth solution of (YM) over $M \times [0,T)$ with $A(0) \in \mathcal{A}^\gothk_E$ a $\gothk$-compatible connection, then $A(t) \in \mathcal{A}^\gothk_E$ for all $t \in [0,T)$.
\end{prop} 
\begin{proof} Recall from \S \ref{sec:splitting} that we have an orthogonal splitting
\begin{equation*}
\Lambda^2 \otimes \mathfrak{so}(E) = (\mathfrak{k} \otimes \mathfrak{so}(E)) \oplus (\mathfrak{k}^{\perp} \otimes \mathfrak{so}(E))
\end{equation*}
under which the curvature of any connection $A \in \mathcal{A}_E$ can be written 
$$F= F^{\mathfrak{k}} + F^{\perp}.$$ 
According to Definition \ref{def:Compatible_Connection}, $A \in \mathcal{A}^\gothk_E$ if and only if $F^\perp = 0.$

For $F = F(t)$ the curvature of a solution of (YM), the evolution formula (\ref{ymcurvatureweitz}) yields
	\begin{equation}\label{eq:gperp_curvature_evolution}
	\begin{split}
	\frac{1}{2} \del_t |F^{\perp}|^2 & =  \langle F^{\perp}, \partial_t F \rangle \\
	& =  \langle F^{\perp}, - \nabla^* \nabla F  + \llbracket F , F \rrbracket + \sum_{a} \rho_a \llbracket \varepsilon_a , \llbracket \varepsilon_a , F \rrbracket \rrbracket \rangle .
	\end{split}
	\end{equation}
Here we have taken $\varepsilon_a \in n(\gothg) \subset\mathfrak{k},$ by the Ambrose-Singer theorem. 
In particular, since both $\gothk$ and $\gothk^{\perp}$ are invariant under the adjoint action of $\mathfrak{k}$, we have
$$\sum_{a} \rho_a  \llbracket \varepsilon_a , \llbracket \varepsilon_a , F \rrbracket \rrbracket  =  \underbrace{\sum_{a} \rho_a \llbracket \varepsilon_a , \llbracket \varepsilon_a , F^{\gothk} \rrbracket \rrbracket}_{ \in \gothk}  +  \underbrace{\sum_{a} \rho_a \llbracket \varepsilon_a , \llbracket \varepsilon_a , F^{\perp} \rrbracket \rrbracket}_{ \in \gothk^\perp} .$$
Only the second term survives in the inner product with $F^{\perp}$ in (\ref{eq:gperp_curvature_evolution}). Next, we write
$$\llbracket F , F \rrbracket = \underbrace{\llbracket F^{\gothk} , F^{\gothk} \rrbracket}_{ \in \gothk} + 2\llbracket F^{\gothk} , F^{\perp} \rrbracket + \llbracket F^{\perp} , F^{\perp} \rrbracket.$$
Likewise, the first term will not contribute to (\ref{eq:gperp_curvature_evolution}). Also note that
$$\langle F^{\perp}, \nabla^* \nabla F \rangle = \langle F^{\perp}, \nabla^* \nabla F^{\perp} \rangle$$ 
since the Levi-Civita connection preserves the splitting (\ref{eq:gperp_curvature_evolution}), by the torsion-free assumption.

Returning to (\ref{eq:gperp_curvature_evolution}), we now have
\begin{equation}\nonumber
\begin{split}
\frac{1}{2} \del_t |F^{\perp}|^2 & =  - \langle F^{\perp}, \nabla^* \nabla F^{\perp} \rangle  - 	
2 \langle F^{\perp} , \llbracket F^{\gothk} , F^{\perp} \rrbracket \rangle - 
\langle  F^{\perp} , \llbracket F^{\perp} , F^{\perp} \rrbracket \rangle - 
\sum_{a} \langle F^{\perp} , \rho_a \llbracket \varepsilon_a , \llbracket \varepsilon_a , F^{\perp} \rrbracket \rrbracket \rangle .
\end{split}
\end{equation}
Since the flow is smooth over $[0,T),$ for each $0 < \tau < T,$ there is a constant $C_{\tau}$ such that $\|F(t) \|_{L^\infty} \leq C_{\tau}$ for $t \in [0, \tau]$. Hence, in $[0, \tau],$ we have
	\begin{equation}
	\begin{split}
	\frac{1}{2} \del_t |F^{\perp}|^2 & \leq  - \langle F^{\perp} , \nabla^* \nabla F^{\perp} \rangle  + \left( C_M + C_\tau \right) | F^{\perp}|^2.
	\end{split}
	\end{equation}
Integrating over $M,$ we obtain
	\begin{equation}\nonumber
	\begin{split}
	\del_t \| F^{\perp} \|_{L^2}^2 & \leq  - 2 \| \nabla F^{\perp} \|_{L^2}^2  + C \| F^{\perp} \|_{L^2}^2 \leq C \| F^{\perp} \|_{L^2}^2.
	\end{split}
	\end{equation}
Multiplying by $e^{-Ct}$ and integrating in time yields
$$\| F^{\perp}(t) \|^2 \leq C \| F^{\perp}(0) \|^2 = 0$$
for $t \in [0,\tau].$ Since $\tau$ was arbitrary, we are done.
\end{proof}

\subsection{Evolution equation on a K\"ahler manifold}\label{ss:Kahler_Evolution}

	According to \S \ref{ss:Kahler}, in the K\"ahler case, we have a splitting of the $2$-forms 
\begin{equation*}
\Lambda^2 \cong \mathfrak{su}(k) \oplus \mathfrak{u}(1) \oplus \mathfrak{u}(k)^{\perp}
\end{equation*}
	where
	$$\mathfrak{su}(k) \cong \Lambda^2_- = \Lambda^{1,1}_{0} , \qquad \mathfrak{u}(1) \cong \Lambda^2_+ = \LA \omega \RA.$$
For a compatible solution of (YM), the curvature splits accordingly:
$$F(t) = F^{\gothk}(t) = F^-(t) + F^+(t).$$
The Weitzenbock formula (\ref{ymcurvatureweitz}) reads
\begin{equation}\label{kahlerfplus1}
\begin{split}
 \del_t F^+ & = - \pi_+ \Delta_A F \\
	& = \pi_+ \left( - \nabla^* \nabla F  + \llbracket F , F \rrbracket + \sum_{a} \rho_a \llbracket \varepsilon_a , \llbracket \varepsilon_a , F \rrbracket \rrbracket \right)
	\end{split}
	\end{equation}
where $\varepsilon_a$ takes values in $\mathfrak{u}(k).$ 

Note that $F^+$ lies in $\gothu(1),$ which is the center of $\gothu(k),$ while $F^-$ lies in $\mathfrak{su}(k),$ which is normalized by $\gothu(k).$ Therefore $\llbracket \varepsilon_a, F^+ \rrbracket = 0$ and
$$\sum_{a} \rho_a \llbracket \varepsilon_a , \llbracket \varepsilon_a , F \rrbracket \rrbracket = \sum_{a} \rho_a \llbracket \varepsilon_a , \llbracket \varepsilon_a , F^{-} \rrbracket \rrbracket \in \Omega^2_- \left( \gothso(E) \right).$$
For the same reason, we have
$$\llbracket F , F \rrbracket = \llbracket F^- , F^- \rrbracket \in \Omega^2_- \left( \gothso(E) \right).$$
Writing
$$F^\omega = F^+ = \left( \Lambda_\omega F \right) \omega$$
we obtain from (\ref{kahlerfplus1}) the well-known evolution equation
\begin{equation}\label{kahlerfplus}
\del_t F^\omega = - \nabla^* \nabla F^\omega
\end{equation}
with no quadratic curvature terms.

\subsection{Evolution equations on a $\rG_2$-manifold}\label{ss:G2_Evolution}

We now describe the case of a $\rG_2$-manifold, where the curvature evolution turns out to be more complex than in the 4-dimensional or K\"ahler cases. These equations will be analyzed in future work. 

A $\rG_2$-structure $\phi$ on $M$ determines several binary operations on differential forms (see \cite{bryantremarksong2}, \cite{karigiannisflowsofg2}, or \cite{Salamon2010}). The octonionic cross-product $\times$ 
is defined by the requirement
$$\phi(\alpha^\#, \beta^\#, \delta^\#) = g \left( \alpha \times \beta, \delta  \right) \qquad \forall \alpha, \beta, \delta \in \Omega^1.$$
Here $\alpha^\#$ is the dual tangent vector to $\alpha$ under the metric $g = g_\phi$ defined by the positive 3-form $\phi.$ A section $\gamma \in \Omega^2_{14}$ acts on sections $\alpha \in \Omega^1$ by
$$\alpha \mapsto \gamma \cdot \alpha.$$
This action is pointwise equivalent to the standard representation of $\gothg_2$ on $\R^7.$ We may also define the projected wedge product
$$\wedge_{14} : \Omega^1 \otimes \Omega^1 \stackrel{\wedge}{\longrightarrow } \Omega^2 \rightarrow \Omega^2_{14}$$
which has the explicit formula (\ref{projectedwedge}) below.
Each of these operations may be extended to $\gothso(E)$-valued forms by coupling with the bracket $\LB, \RB.$

These operations are related to the commutator $\llbracket, \rrbracket$ on 2-forms, given by (\ref{doublebracket}), as follows.

\begin{lemma}\label{lemma:productlemma}
The map
	\begin{equation*}
	\begin{split}
	& \Omega^1 \stackrel{\sim}{\longrightarrow} \Omega^2_7 \\
	& \alpha \mapsto \alpha \intprod \phi
	\end{split}
	\end{equation*}
is equivariant under the action of $\Omega^2_{14},$ \textit{i.e.} for $\gamma \in \Omega^2_{14},$ we have
\begin{equation}\label{productlemma:714}
\left( \gamma \cdot \alpha \right) \intprod \phi = \llbracket \gamma, \alpha \intprod \phi \rrbracket.
\end{equation}
For $\alpha, \beta \in \Omega^1,$ we have
\begin{equation}
\begin{split}
\label{productlemma:pi7and14} \pi_7 \llbracket \alpha \intprod \phi , \beta \intprod \phi \rrbracket & = \left( \alpha \times \beta\right) \intprod \phi \\
\pi_{14} \llbracket \alpha \intprod \phi , \beta \intprod \phi \rrbracket & = -3 \alpha \wedge_{14} \beta. 
\end{split}
\end{equation}
\end{lemma}
\begin{proof} 
The 2-forms $\gamma \in \Omega^2_{14}$ are precisely those which preserve $\phi,$ \textit{i.e.}
\begin{equation*}
0 = \gamma_{ji} \phi_{ik\ell} + \gamma_{ki} \phi_{ji\ell} + \gamma_{\ell i} \phi_{jki}.
\end{equation*}
Rearranging yields
\begin{equation*}
\gamma_{ij} \phi_{ik\ell} = \gamma_{ki} \phi_{ji\ell} - \gamma_{\ell i} \phi_{jik}.
\end{equation*}
Contracting with $\alpha_j,$ we obtain
\begin{equation*}
\gamma_{ij} \alpha_j \phi_{ik\ell} = \gamma_{ki} \alpha_j \phi_{ji\ell} - \gamma_{\ell i} \alpha_j \phi_{jik}
\end{equation*}
which is (\ref{productlemma:714}).

To prove (\ref{productlemma:pi7and14}), we first claim that
\begin{equation}\label{productlemma:bracket23}
\llbracket \alpha \intprod \phi , \beta \intprod \phi \rrbracket = 2 \left( \alpha \times \beta\right) \intprod \phi  - 3 \alpha \wedge \beta.
\end{equation}
To establish (\ref{productlemma:bracket23}), it suffices to work with the standard $\rG_2$-structure on $\R^7,$ given by
\begin{equation*}
\phi = e^{123} - e^{145} - e^{167} - e^{246} + e^{257} - e^{347} - e^{356}.
\end{equation*}
Since $\rG_2$ acts transitively on orthonormal pairs, we may take $\alpha = e_1$ and $\beta = e_2,$ so that
\begin{equation*}
\begin{split}
e_1 \intprod \phi = e^{23} - e^{45} - e^{67}, \qquad e_2 \intprod \phi = -e^{13} - e^{46} + e^{57}, \qquad 
e_1 \times e_2 = e_3.
\end{split}
\end{equation*}
We calculate
\begin{equation*}
\begin{split}
\llbracket e_1 \intprod \phi, e_2 \intprod \phi \rrbracket & = - \llbracket e^{23}, e^{13} \rrbracket + \llbracket e^{45}, e^{46} \rrbracket - \llbracket e^{45}, e^{57} \rrbracket + \llbracket e^{67}, e^{46} \rrbracket - \llbracket e^{67}, e^{57} \rrbracket \\
& = - e^{12} - 2 e^{56} - 2e^{47} \\
& = -3 e^{12} + 2 \left( e^{12} - e^{56} - e^{47} \right) \\
& = -3 e^{12} + 2 e_3 \intprod \phi.
\end{split}
\end{equation*}
By $\rG_2$-equivariance and linearity, this proves the general formula (\ref{productlemma:bracket23}).

Writing $\left( \omega \intprod \phi \right)_{k} = \omega_{ij} \phi_{ijk}$ for a 2-form $\omega,$ the identities
\begin{equation*}
\left(\alpha \intprod \phi \right) \intprod \phi = 6 \alpha, \qquad \left( \alpha \wedge \beta \right) \intprod \phi = 2 \alpha \times \beta
\end{equation*}
are easily checked as above. Hence, the projection operator $\Omega^2 \to \Omega^2_7$ is given by
\begin{equation*}
\pi_7 (\omega) = \frac16 \left( \omega \intprod \phi \right) \intprod \phi.
\end{equation*}
We therefore have
$$\alpha \wedge_7 \beta = \frac{1}{3} \left( \alpha \times \beta \right) \intprod \phi$$
and
\begin{equation}\label{projectedwedge}
\alpha \wedge_{14} \beta = \alpha \wedge \beta - \frac{1}{3} \left( \alpha \times \beta \right) \intprod \phi.
\end{equation}
The desired equations (\ref{productlemma:pi7and14}) now follow by applying $\pi_7$ and $\pi_{14}$ to (\ref{productlemma:bracket23}).
\end{proof}

\begin{prop}\label{prop:g2weitz}
	Let $(M ,\phi)$ be a $\rG_2$-holonomy manifold and $A$ a connection whose curvature $F = F_A$ we write as
	$$F= f^7 \intprod \phi + F^{14}.$$ 
	Then, for sections $\alpha \intprod \phi \in \Omega^2_7 \left( \mathfrak{so}(E) \right)$ and $\omega \in \Omega^2_{14}\left( \mathfrak{so}(E) \right),$ the Weitzenb\"ock formula (\ref{weitz2brackets}) may be rewritten
	\begin{equation*}
	\begin{split}
	(a) \quad & \Delta_A (\alpha \intprod \phi) = \big( \nabla^* \nabla \alpha - \LB f^7 \times \alpha \RB - \LB F^{14} \cdot \alpha \RB \big) \intprod \phi + 3 \LB f^7 \wedge_{14} \alpha \RB \\
	(b) \quad & \Delta_A \omega = \nabla^* \nabla \omega - \llbracket F^{14} , \omega \rrbracket - \LB  \omega \cdot f^{7} \RB \intprod \phi - \sum_{a} \rho_a \llbracket \varepsilon_a , \llbracket \varepsilon_a, \omega \rrbracket \rrbracket.
	\end{split}
	\end{equation*}
\end{prop}
\begin{proof} Part ($a$) follows from Lemma \ref{lemma:productlemma} and (\ref{weitz2brackets}), where we claim that the Riemann curvature term vanishes. Note that since $\varepsilon_a \in \Omega^2_{14},$ the operator $R \#$ preserves $\Omega^2_7.$ Hence, to check the vanishing, it suffices to calculate as follows: 
	\begin{equation*}
	\begin{split}
	\LA \alpha \intprod \phi , R \# (\beta \intprod \phi) \RA = - \LA \alpha \intprod \phi , \sum_{a} \rho_a \llbracket \varepsilon_a , \llbracket \varepsilon_a , \beta \intprod \phi \rrbracket \rrbracket \RA
	& = - \sum_{a} \rho_a \LA \llbracket \alpha \intprod \phi , \varepsilon_a  \rrbracket , \llbracket \varepsilon_a , \beta \intprod \phi  \rrbracket \RA \\
	& = \sum_{a} \rho_a \LA \left( \varepsilon_a \cdot \alpha \right) \intprod \phi , \left( \varepsilon_a \cdot \beta \right) \intprod \phi \RA \\
	& = 3 \sum_{a} \rho_a \LA \varepsilon_a \cdot \alpha , \varepsilon_a \cdot \beta \RA \\
	& = 3\sum_{a} R_{ikjk} \alpha_i \beta_j = 0.
	\end{split}
	\end{equation*}
We have used (\ref{productlemma:714}) in the second line, and Ricci flatness in the last line.
	
	Part ($b$) simply restates (\ref{weitz2brackets}) using (\ref{productlemma:714}), while paying heed to (\ref{commutatativity}).
\end{proof}


\begin{cor}\label{cor:g2commutator} On a $\rG_2$-holonomy manifold, the commutator
	\begin{equation*}
	\LB \pi_{14}, \Delta_A \RB = - \LB \pi_7 , \Delta_A \RB
	\end{equation*}
	on $\Omega^2(\mathfrak{so}(E)) = \Omega^2_7(\mathfrak{so}(E)) \oplus \Omega^2_{14}(\mathfrak{so}(E))$ is given by the endomorphism
	\begin{equation*}
	\begin{pmatrix}
	\alpha \intprod \phi \\
	\omega
	\end{pmatrix}
	\longmapsto
	\begin{pmatrix} \LB \omega \cdot f^7 \RB \intprod \phi \\
	3 \LB f^7 \wedge_{14} \alpha \RB
	\end{pmatrix} 
	\end{equation*}
	where $f^7 \intprod \phi = F_A^7.$
	In particular, 
	$\Delta_A$ preserves the splitting of $\Omega^2(\mathfrak{so}(E))$ if $A$ is a $\rG_2$-instanton.
\end{cor}

\begin{cor}\label{cor:G2_Evolution}
	Let $A(t)$ be a smooth solution of the Yang-Mills flow on a $\rG_2$-manifold. Writing the curvature as 
	$$F(t) = F_{A(t)} = f^7(t)\intprod \phi + F^{14} (t)$$ 
	we have the evolution equations
	\begin{equation*}
	\begin{split}
	(a) \quad & \left( \frac{\partial}{\partial t} + \nabla^* \nabla \right) f^7(t) = \LB f^7 \times f^7 \RB + 2 \LB F^{14} \cdot f^7 \RB \\
	(b) \quad & \left( \frac{\partial}{ \partial t} + \nabla^* \nabla \right) F^{14}(t) = \llbracket F^{14} , F^{14} \rrbracket - 3 \LB f^7 \wedge_{14} f^7 \RB + \sum_a \rho_a \llbracket \varepsilon_a , \llbracket \varepsilon_a, F^{14} \rrbracket \rrbracket.
	\end{split}
	\end{equation*}
\end{cor}

\section{Extended monotonicity formula}\label{sec:Monotonicity}

\subsection{Hamilton's formula}

Recall from \cite{Waldron2016} the pointwise energy identity for Yang-Mills flow: 
\begin{equation}\label{pointwiseenergy}
\begin{split}
\frac{1}{2} \frac{\partial}{\partial t} |F|^2 + |D^*F|^2 & = \nabla^i \nabla^j S_{ij}. 
\end{split}
\end{equation}
Here $S_{ij}$ is the stress-energy tensor
\begin{equation}\label{stressenergydef}
S_{ij} = g^{k\ell} \LA F_{ik}, F_{j \ell} \RA - \frac{1}{4} g_{ij} g^{k\ell}g^{mn}\LA F_{km}, F_{\ell n} \RA.
\end{equation}
The identity (\ref{pointwiseenergy}) allows for efficient proofs of the standard monotonicity formula due to Hamilton \cite{Hamilton1993}, Corollary \ref{cor:HamiltonMonotonicity},
as well as a generalization in the presence of a torsion-free $N(\rG)$-structure, Theorem \ref{thm:mainthm}. The following evolution formula (\ref{generalpointwisemonotonicity}) is convenient for both purposes. 

\begin{prop}\label{generalpointwiseprop} 
Let $M$ be an oriented Riemannian manifold and $u(x,\tau)$ a smooth function on $M \times \left( 0 , \infty \right).$ Let $\tau = \tau(t)$ be a smooth decreasing function which is positive on $\LB 0, T \right),$ and put
$$\gamma = \gamma(t) = \sqrt{-\tau'(t) }.$$
Assume that $A(t)$ is a smooth solution of (YM) on $M \times [0,T),$ and denote its curvature by $F = F_{A(t)}.$ The following formula holds:
\begin{samepage}
\begin{equation}\label{generalpointwisemonotonicity}
\begin{split}
\frac{1}{2} \frac{\partial}{\partial t} \left( u |F|^2 \right) + u \left| D^*F - \gamma \frac{\nabla u}{u} \intprod F \right|^2
 & = \gamma^2 \nabla^i \nabla^j \left( u S_{ij} \right) \\
 & + \left(1 - \gamma^2 \right) \nabla^i \left( u \nabla^j S_{ij} \right)  - \left( 1 - \gamma \right)^2 \nabla^i u \nabla^j S_{ij} \\
 & - \frac{\gamma^2}{2} \left( \frac{\partial u}{\partial \tau } + \Delta u - \frac{2 u}{\tau} \right) |F|^2  \\
 & - \gamma^2 \left( \nabla^i \nabla^j u - \frac{\nabla^i u \nabla^j u}{u} + \frac{u}{2 \tau} g^{ij} \right) g^{k \ell} \LA F_{ik} , F_{j \ell} \RA.
 \end{split}
\end{equation}
Here $\Delta$ is the geometer's (positive) Laplacian.
\end{samepage}
\end{prop}
\begin{proof} Multiplying (\ref{pointwiseenergy}) by $u$ and using the Leibniz rule, we obtain
\begin{equation}\label{generalpointwise1}
\frac12 \frac{\partial}{\partial t} \left( u |F|^2 \right) + u |D^*F|^2 = \nabla^i \left( u \nabla^j S_{ij} \right) - \nabla^i u \nabla^j S_{ij} -\frac{\gamma^2}{2} \frac{\partial u}{\partial \tau} |F|^2.
\end{equation}
Recalling that $|F|^2 = \frac{1}{2}g^{ij}g^{k\ell} \LA F_{ik}, F_{j\ell} \RA$ and using the Leibniz rule again, we have
\begin{equation}\label{generalpointwiseidentity}
\begin{split}
\nabla^i \left( \nabla^j u S_{ij} \right) & =  \nabla^i \nabla^j u S_{ij} + \nabla^j u \nabla^i S_{ij} \\
& = \nabla^i \nabla^j u g^{k \ell} \LA F_{ik}, F_{j \ell} \RA + \frac{1}{2} \Delta u |F|^2 + \nabla^ju \nabla^i S_{ij}.
\end{split}
\end{equation}
Multiplying (\ref{generalpointwiseidentity}) by $\gamma^2$ and combining with (\ref{generalpointwise1}), we have
\begin{equation}\label{preidentity1}
\begin{split}
\frac{1}{2} \frac{\partial}{\partial t} \left( u |F|^2 \right) + u \left| D^*F \right|^2
 & = \nabla^i \left( u\nabla^j S_{ij} + \gamma^2 \nabla^j u S_{ij} \right) -\left( \gamma^2 + 1 \right) \nabla^i u \nabla^j S_{ij} \\
 & - \frac{\gamma^2}{2} \left( \frac{\partial u}{\partial \tau } + \Delta u \right) |F|^2  \\
 & - \gamma^2 \nabla^i \nabla^j u  g^{k \ell} \LA F_{ik} , F_{j \ell} \RA.
 \end{split}
\end{equation}
Next, we add $\gamma^2 \frac{\nabla^i u \nabla^j u}{u} g^{k \ell} \LA F_{ik} , F_{j \ell} \RA$ to both sides of (\ref{preidentity1}), to obtain the LHS
$$\frac{1}{2} \frac{\partial}{\partial t} \left( u |F|^2 \right) + u \left( \left| D^*F \right|^2 + \gamma^2 \frac{\nabla^i u \nabla^j u}{u^2} g^{k \ell} \LA F_{ik} , F_{j \ell} \RA \right).$$
To complete the square, we further add
$$- 2\gamma u g^{k\ell} \LA D^*F_k, \frac{\nabla^i u}{u} F_{i\ell} \RA = 2\gamma \nabla^i u \nabla^j S_{ij}$$
obtaining
\begin{equation}\label{eq:intermediate1}
\begin{split}
\frac{1}{2} \frac{\partial}{\partial t} \left( u |F|^2 \right) & + u \left| D^*F - \gamma \frac{\nabla u}{u} \intprod F \right|^2 \\
 & = \nabla^i \left( u\nabla^j S_{ij} + \gamma^2 \nabla^j u S_{ij} \right) - \left(\gamma - 1 \right)^2 \nabla^i u \nabla^j S_{ij} \\
 & - \frac{\gamma^2}{2} \left( \frac{\partial u}{\partial \tau } + \Delta u \right) |F|^2  \\
 & - \gamma^2 \left( \nabla^i \nabla^j u - \frac{\nabla^i u \nabla^j u}{u} \right)  g^{k \ell} \LA F_{ik} , F_{j \ell} \RA.
 \end{split}
\end{equation}
Finally, we make the substitution
$$\nabla^i \left( u\nabla^j S_{ij} + \gamma^2 \nabla^j u S_{ij} \right) = \gamma^2 \nabla^i \nabla^j \left( u S_{ij} \right) + \left( 1 - \gamma^2 \right) \nabla^i \left(u \nabla^j S_{ij} \right) $$
and add 
$\gamma^2$ times the identity
$$0 = \frac{u}{\tau} |F|^2 - \frac{u}{2\tau} g^{ij} g^{k \ell} \LA F_{ik} , F_{j \ell} \RA$$
to (\ref{eq:intermediate1}), to obtain the desired formula (\ref{generalpointwisemonotonicity}).
\end{proof}

\begin{cor}[Hamilton \cite{Hamilton1993}]\label{cor:HamiltonMonotonicity} Assume that $M$ is Ricci-parallel with nonnegative sectional curvatures, and compact. Let $v$ be a solution of the backwards heat equation $\left( \partial_t - \Delta \right) v = 0$ on $M \times \LB 0 , T \right).$ Then, for $0 \leq t_1 \leq t_2 < T,$ we have 
\begin{equation*}
\begin{split}
(T - t_2)^2 \int_M | F(x, t_2) |^2 v(x , t_2) \, dV & + 2 \int_{t_1}^{t_2} \left( T -t \right)^2 \int_M \left| D^*F - \frac{\nabla v}{v} \intprod F \, \right|^2 v \, dV dt \\
& \leq (T - t_1)^2 \int_M | F(x, t_1) |^2 v(x , t_1) \, dV.
\end{split}
\end{equation*}
\end{cor}
\begin{proof} Let $u(x,\tau) = \tau^2 v(x, T - \tau)$ and $\tau(t) = T - t,$ hence $\gamma \equiv 1,$ in the previous Proposition. The second and third lines on the RHS of (\ref{generalpointwisemonotonicity}) vanish by definition. The expression in the last line of (\ref{generalpointwisemonotonicity}) is Hamilton's matrix Harnack quantity \cite{Hamilton1993_2}; 
hence, the last term is nonpositive under the stated assumptions on $M.$ We obtain the pointwise inequality
$$\frac{1}{2} \frac{\partial}{\partial t} \, (T - t)^2 v |F|^2 + (T - t)^2 v \left| D^*F - \frac{\nabla v}{v} \intprod F \, \right|^2
\leq (T - t)^2 \nabla^i \nabla^j \left( v S_{ij} \right).$$
The result follows by integrating over $M \times \LB t_1, t_2 \RB.$
\end{proof}

\subsection{Splitting of the stress-energy tensor}\label{sec:Stress_Energy}

In the presence of an $N(G)$-structure as in \S \ref{sec:splitting}, the stress-energy tensor (\ref{stressenergydef}) may be decomposed as
$$S_{ij} = \sum_\alpha \tilde{S}^\alpha_{ij}$$
where
\begin{equation*}
\begin{split}
\tilde{S}^{\alpha}_{ij} & = g^{k\ell} \LA F^{\alpha}_{ik}, F_{j \ell} \RA - \frac{1}{4} g_{ij} g^{k\ell}g^{mn}\LA F^{\alpha}_{km}, F^{\alpha}_{\ell n} \RA \quad (\mbox{no sum on $\alpha$}).
\end{split}
\end{equation*}

\begin{prop}\label{prop:Divergence_Proportional}
Given a torsion-free $N(G)$-structure as in \S \ref{sec:splitting}, define $\kappa_\alpha$ by (\ref{kappahat}). Then
$$ \nabla^i S_{ij} = \kappa_\alpha \nabla^k \tilde{S}^\alpha_{kj}$$ 
where $i,k,$ and $\alpha$ are summed over.
\end{prop}
\begin{proof}
Fix an index $\alpha.$ Computing in normal coordinates, we have
\begin{equation}\label{stressdivergence}
\begin{split}
\nabla_i \tilde{S}^\alpha_{ij} & = \LA \nabla_i F^\alpha_{ik}, F_{j k} \RA + \LA F^\alpha_{ik}, \nabla_i F_{j k} \RA - \frac{1}{4} \nabla_j \LA F^\alpha_{k \ell}, F^\alpha_{k \ell} \RA.
\end{split}
\end{equation}
Using the Bianchi identity, the second term of the RHS may be simplified as follows:
\begin{equation}\label{stressbianchi}
\begin{split}
\LA F^\alpha_{ik}, \nabla_i F_{j k} \RA & = \LA F^\alpha_{ik}, \nabla_j F_{i k} \RA + \LA F^\alpha_{ik}, \nabla_k F_{j i} \RA \\
& = \frac{1}{2} \LA F^\alpha_{ik}, \nabla_j F_{i k} \RA.
\end{split}
\end{equation}
Since the Levi-Civita connection preserves the orthogonal splitting of the 2-forms, we have
$$\LA F^\alpha_{ik}, \nabla_j F^{\beta}_{i k} \RA = 0 \quad (\beta \neq \alpha).$$
Hence (\ref{stressbianchi}) becomes
$$\LA F^\alpha_{i k}, \nabla_i F_{j k} \RA = \frac{1}{2} \LA F^\alpha_{ik}, \nabla_j F^\alpha_{i k} \RA = \frac{1}{4} \nabla_j \LA F^\alpha_{k \ell}, F^\alpha_{k \ell} \RA \quad (\mbox{no sum on $\alpha$}).$$
Therefore the last two terms in (\ref{stressdivergence}) cancel, and we left with
\begin{equation*}
\begin{split}
\nabla_i \tilde{S}^\alpha_{ij} & = \LA \nabla_i F^\alpha_{i k}, F_{j k} \RA \\
& = \LA \left( D^*F^\alpha \right)_k, F_{kj} \RA.
\end{split}
\end{equation*}

On the other hand, we also have
\begin{equation}
\nabla_i S_{ij} = \LA D^*F_k, F_{kj} \RA.
\end{equation}
From Lemma \ref{lemma:dstarf}$i,$ we obtain
\begin{equation*}
\begin{split}
\nabla_i S_{ij} & = \LA \kappa_\alpha \left( D^*F^\alpha \right)_k, F_{kj} \RA \\
& = \kappa_\alpha \nabla_k \tilde{S}^\alpha_{kj}
\end{split}
\end{equation*}
where $\alpha$ is now summed over, as desired.
\end{proof}

\begin{rmk} One may define symmetric tensors
\begin{equation*}
S^\alpha_{ij} = \frac{1}{2} \left( \tilde{S}^\alpha_{ij} + \tilde{S}^\alpha_{ji} \right).
\end{equation*}
If $(M,g)$ is Ricci flat, 
we compute
\begin{equation*}
\begin{split}
\nabla_i \nabla_j S^{\alpha}_{ij} & = \nabla_j\nabla_i \tilde{S}^{\alpha}_{ij} + \frac{1}{2} \left( \LB \nabla_i, \nabla_j \RB \tilde{S}^{\alpha}_{ij} \right) \\
& =  \nabla_j \nabla_i \tilde{S}^{\alpha}_{ij} + \frac{1}{2} \left( R_{ijik} \tilde{S}^{\alpha}_{kj} - R_{ijjk} \tilde{S}^{\alpha}_{ki} \right) \\
& = \nabla_j \nabla_i \tilde{S}^{\alpha}_{ij}. 
\end{split}
\end{equation*}
Substituting into Proposition \ref{prop:Divergence_Proportional} gives
\begin{equation*}
\nabla_i \nabla_j S_{ij} = \kappa_\alpha \nabla_i \nabla_j S^{\alpha}_{ij}.
\end{equation*}
\end{rmk}

\subsection{Extended formula}

This section proves our extension of Hamilton's monotonicity formula. The new feature is that the time interval $t_2 - t_1$ is allowed to be longer than the square of the radius; in particular, it need not tend to zero with $R.$ 

For simplicity, we will fix $\beta = 0,$ so that
$$F^- = F^\beta, \qquad F^+ = \sum_{\alpha \geq 1} F^\alpha$$
and $\kappa_0 = 0,$ per (\ref{kappahat}). Also let
\begin{equation}\label{kappadef}
\kappa = \sum_\alpha \left| \kappa_\alpha \right|.
\end{equation}
The results of this and the next section hold for any choice of $\beta$ with $\lambda_\beta \neq 0,$ by replacing $F^+$ with $\sum_{\alpha \neq \beta} F^\alpha.$

Let $x_1 \in M$ with $\inj(M, x_1) \geq \rho_1 > 0.$ Here, we let $\inj(M, x_1)$ denote the maximal radius of a normal geodesic ball centered at $x_1.$ Fix a smooth cutoff function $\varphi(r)$ supported on the unit interval, with $\varphi(r) \equiv 1$ on $\LB 0 , 1/2 \RB,$ and let
$$\varphi_{x_1, \rho_1}(y) = \varphi \left( \frac{d \left(x_1, y \right)}{\rho_1} \right).$$
\begin{defn}\label{def:entropy}
Given a connection $A,$ define the weighted energy functional 
\begin{equation*}
\Phi_{x_1, \rho_1}(A; R, x) = \frac{R^{4-n}}{(4\pi)^{n/2}} \int_M |F_A(y)|^2 \exp \left(-\left(\frac{d \left( x, y \right)}{2R} \right)^2 \right) \varphi_{x_1, \rho_1}(y) \, dV_y. 
\end{equation*}
For a solution $A(t)$ of (YM), write
\begin{equation*}
\Phi_{x_1, \rho_1}(R,x,t) = \Phi_{x_1, \rho_1}(A(t); R, x).
\end{equation*}
\end{defn}

\begin{lemma}\label{gaussianlemma} Let $x \in B_{\rho_1}(x_1),$ and define
\begin{equation}\label{udef}
u(y,\tau) = \frac{\tau^{2-\frac{n}{2}}}{(4\pi )^{n/2}} \exp - \left( \frac{ d \left( x, y \right)^2 }{4\tau} \right) .
\end{equation}
Writing $r = d \left(x , y \right),$ we have the following on $B_{\rho_1}(x_1):$
\begin{equation*}
\begin{split}
(a) \qquad & | \nabla^i \nabla^j u | \leq \left(C_0 r^2 + \sqrt{n} + \frac{r^2}{2 \tau} \right) \frac{u}{2 \tau} \\
(b) \qquad & \left| \left( \frac{\pd}{\pd \tau} + \Delta - \frac{2}{\tau} \right) u \right| \leq C_0 r u \\
(c) \qquad & \left| \nabla^i \nabla^j u - \frac{\nabla^i u \nabla^j u}{u} + \frac{ u }{2 \tau} g^{ij} \right| \leq C_0 \frac{r^2}{\tau} u.
\end{split}
\end{equation*}
Here $C_0 \geq 0$ depends on $g$ and can be made arbitrarily small by rescaling.
\end{lemma}
\begin{proof} 
Let $\lbrace X^i \rbrace_{i=1}^n$ be normal coordinates centered at $x.$ Then, the radial vector field
$$X = X^i \frac{\pd}{\pd X^i}$$
satisfies
\begin{equation}
\nabla^i X^j = g^{ij} + h^{ij}
\end{equation}
for a symmetric tensor $h^{ij}$ with
\begin{equation}\label{cgdefinition}
|h(y)|_g \leq C_0 d \left( x , y \right)^2 \mbox{ on } B_{\rho_1}(x_1).
\end{equation}
The constant $C_0$ (as any other) may increase in each subsequent appearance.

The above inequalities 
may be verified using the following formulae, which are valid for any radial function $\psi = \psi(r):$
$$\nabla^j \psi = \frac{X^j}{r} \psi'(r), \quad \quad \nabla^i \nabla^j \psi = \left( g^{ij} + h^{ij} \right) \frac{\psi'(r)}{r} + \frac{X^i X^j}{r^2} \left( \psi''(r) - \frac{\psi'(r)}{r} \right).$$
For ($a$), we obtain
\begin{equation*}
\begin{split}
\nabla^i \nabla^j u & = \frac{\tau^2}{\left( 4 \pi \tau \right)^{n/2} } \left( \left( g^{ij} + h^{ij} \right) \frac{-1}{2\tau} + \frac{X^i X^j}{4 \tau^2} \right) e^{- \frac{r^2}{4\tau} } \\
& = \left( - \left( g^{ij} + h^{ij} \right) + \frac{X^i X^j}{2 \tau} \right) \frac{u}{2 \tau}.
\end{split}
\end{equation*}
This yields ($a$), after applying $|\cdot|$ and (\ref{cgdefinition}).

The estimates ($b$) and ($c$) are proved similarly.
\end{proof}

\begin{thm}\label{thm:mainthm} 
Let $M$ be a Riemannian manifold of dimension $n \geq 4,$ equipped with a torsion-free $N(G)$-structure as in \S \ref{sec:splitting}. Fix $x, x_1 \in M,$ with $\inj(M, x_1) \geq \rho_1 > 0$ and $d(x, x_1) \leq \rho_1 / 4.$

Let $A(t)$ be a $\gothk$-compatible solution of (YM) on $M \times \LB 0 , T \right),$ per Definition \ref{def:Compatible_Connection}. For $F(t) = F_{A(t)},$ write
\begin{equation}\label{kdefinition}
K(t) = \| F^+ (t)  \|_{L^\infty(B_{\rho_1}(x_1))}
\end{equation}
\begin{equation*}
\begin{split}
\Phi(R , t) & = \Phi_{x_1, \rho_1}(R, x, t), \quad \quad E = \sup_{ 0 \leq t < T} \int_{B_{\rho_1}(x_1) } |F(t)|^2 \, dV.
\end{split}
\end{equation*}
Let $0 \leq t_1 < t_2 < T$ and $1 \geq R_1 \geq R_2 > 0$
be such that 
$$\gamma = \sqrt{\frac{R_1^2 - R_2^2}{t_2 - t_1}} \leq 1$$
and put $R(t) = \sqrt{R_1^2 - \gamma^2 \left( t - t_1 \right)}.$
Then, the weighted energy obeys an estimate
\begin{equation}\label{mainestimate}
\begin{split}
\Phi(R_2, t_2) & \leq e^{C_0 \gamma^2 (R_1 - R_2) } \Phi(R_1, t_1) + C_1 \left( R_1^2 - R_2^2 \right) E \\
& \qquad + \kappa \left(1 - \gamma \right) \int_{t_1}^{t_2}  K(t) \sqrt{C_2 \Phi ( R(t)  , t ) + C_3 E } \, dt.
\end{split}
\end{equation}
Here $\kappa$ is defined by (\ref{kappadef}), and the constants $C_i$ depend on $\rho_1$ and the geometry of $B_{\rho_1}(x_1)$ (see Remark \ref{rmk:constants} below). 
\end{thm}
\begin{proof} 
Let $u(y, \tau)$ be defined by (\ref{udef}), and 
$$\tau(t) = R(t)^2 = R_2^2 + \gamma^2 \left( t_2 - t \right).$$
Then, writing $ \varphi = \varphi_{x_1, \rho_1}$ for the cutoff function, we have
$$\Phi(R(t),t) = \int_M | F(y,t) |^2 u(y, \tau(t)) \varphi (y) \, dV_y.$$

We apply the formula of Proposition \ref{generalpointwiseprop} with this choice of $u$ and $\tau.$ Integrating by parts against the cutoff $\varphi,$ and using items ($b$) and ($c$) in Lemma \ref{gaussianlemma}, we have
\begin{equation*}
\begin{split}
\frac{1}{2} \frac{d}{dt} \Phi(R(t), t) & \leq \int \left( \gamma^2 \nabla^i \nabla^j \varphi u S_{ij} + (\gamma^2 -1) \nabla^i \varphi u \nabla^j S_{ij} - (1 -\gamma)^2 \varphi \nabla^i u \nabla^j S_{ij} \right) \, dV \\
& \qquad + C_0 \gamma^2 \int \left( r + \frac{r^2}{\tau} \right) u \varphi |F|^2 \, dV
\end{split}
\end{equation*}
where $r = d(x, y).$ Substituting $\nabla^j S_{ij} = \kappa_\alpha \nabla^j \tilde{S}^\alpha_{ji},$ by Proposition \ref{prop:Divergence_Proportional}, and integrating by parts again, we obtain
\begin{equation}\label{mainthm:3part}
\begin{split}
\frac{1}{2} \frac{d}{dt} \Phi(R(t), t) & \leq \int \left( \begin{split} & (1 - \gamma)^2 \varphi \nabla^i \nabla^j u  + \\
& \quad + (1 - \gamma)^2 \left( \nabla^j \varphi \nabla^i u  \right) + (1 - \gamma^2) \left( \nabla^i \nabla^j \varphi u + \nabla^i \varphi \nabla^j u \right)
\end{split} \right) \kappa_\alpha \tilde{S}^\alpha_{ij} \, dV \\
& \qquad + \gamma^2  \int \left( C | \nabla^{(2)} \varphi | + C_0 \left( r + \frac{r^2}{\tau}\right) \varphi \right) u |F|^2 \, dV.
\end{split} 
\end{equation}

We now estimate each term on the RHS of (\ref{mainthm:3part}). Recall that $|\tilde{S}^\alpha_{ij}| \leq C|F^\alpha| |F|,$ so
\begin{equation*}
\begin{split}
| \kappa_\alpha \tilde{S}^\alpha| & \leq C | \kappa_\alpha | |F^\alpha| |F| \\
& \leq C \kappa \, |F^+| |F|
\end{split}
\end{equation*}
since $\kappa_0 = 0$ and $|F^\alpha| \leq |F^+|$ for $\alpha \neq 0.$
Using H\"older's inequality, we may therefore estimate
\begin{equation}\label{mainthm:holder}
 \int (1 - \gamma)^2 \varphi \nabla^i \nabla^j u \kappa_\alpha \tilde{S}^\alpha_{ij} \, dV \leq C \kappa (1-\gamma)^2 \|F^+ \|_{L^\infty} \left( \int \frac{|\nabla^{(2)} u|^2}{u} \varphi \, dV \right)^{1/2} \Phi(R(t), t)^{1/2}.
\end{equation}
By Lemma \ref{gaussianlemma}$a,$ we have
\begin{equation*}
\frac{| \nabla^{(2)} u|^2}{u} \leq C\left(C_0 r^2 + 1 + \frac{r^2}{\tau} \right)^2  \frac{u}{\tau^2} \leq \left( C_0^2 \tau^2 \left( \frac{r^2}{\tau} \right)^2 + C \left( 1 + \frac{r^2}{\tau} \right)^2 \right) G(r,\tau)
\end{equation*}
where $G$ the fundamental solution of the heat equation on Euclidean space. The integral of the RHS is clearly bounded; in fact, since $\tau(t) \leq R_1^2 \leq 1,$ we have
\begin{equation}\label{c2def}
\begin{split}
\int \frac{| \nabla^{(2)} u|^2}{u} \varphi \, dV & \leq C_0^2 R_1^4 + C \\
& \leq C_0^2 + C =: C_2.
\end{split}
\end{equation}
Substituting into (\ref{mainthm:holder}) yields
\begin{equation}\label{mainthm:firstline}
 (1 - \gamma)^2 \int \varphi \nabla^i \nabla^j u \kappa_\alpha \tilde{S}^\alpha_{ij} \, dV \leq  C_2 \kappa (1-\gamma)^2 \|F^+ \|_{L^\infty} \Phi(R(t), t)^{1/2}.
\end{equation}

Next, notice that the second line of the integrand of (\ref{mainthm:3part}) is supported on $B_{\rho_1}(x_1) \setminus B_{\rho_1/2}(x_1),$ where
\begin{equation}\label{usmallness}
u + |\nabla^i u | \leq C e^{- \frac{\rho_1^2}{16\tau(t)} } = C e^{-\frac{\rho_1^2}{16 R(t)^2} }.
\end{equation}
We then have 
\begin{equation}\label{mainthm:secondline}
\begin{split}
\int_M \left( (1 - \gamma)^2 \nabla^j \varphi \nabla^i u \right. & \left. + (1 - \gamma^2) \left( \nabla^i \nabla^j \varphi u + \nabla^i \varphi \nabla^j u \right)  \right) \kappa_\alpha \tilde{S}^\alpha_{ij} dV \\
& \leq C_M \kappa (1 - \gamma) \int_{B_{\rho_1}(x_1) \setminus B_{\rho_1/2}(x_1)} |F^+| |F| e^{-\frac{\rho_1^2}{16R(t)^2}} dV \\
& \leq C_M \kappa (1- \gamma) e^{-\frac{\rho_1^2}{16R(t)^2}}  \Vert F^+ \Vert_{L^{\infty}(B_{\rho_1}(x_1))} E^{1/2}.
\end{split}
\end{equation}
Here $C_M$ depends on $\rho_1$ and the geometry of $B_{\rho_1}(x_1) \subset M.$

For the third line of (\ref{mainthm:3part}), we have
$$\left( C |\nabla^{(2)} \varphi| + C_0 \left( r + \frac{r^2}{\tau}\right) \varphi \right) u \leq C_M e^{-\frac{\rho_1^2}{16 R(t)^2} } + C_0 \varphi \left (1 + \frac{r^2}{\tau} \right) u$$
where (\ref{usmallness}) was used again.
To estimate the second term, we apply the inequality
$$a \left( 1 + \log(b/a) \right) \leq 1 + a \log b.$$
of Hamilton (\cite{Hamilton1993}, Lemma 1.2), with $a = u$ and $b = (4\pi)^{-n/2} \tau^{2 - n/2}.$ This reads
\begin{equation}
\begin{split}
\left( 1 + \frac{r^2}{4\tau}\right) u & \leq 1 - u \log \left( (4\pi)^{n/2} \tau^{n/2 - 2} \right) \\
& \leq 1 - \left( \frac{n}{2} - 2 \right) u \log \tau
\end{split}
\end{equation}
since $u \geq 0.$ Hence, the term in the third line of (\ref{mainthm:3part}) becomes
\begin{equation}\label{mainthm:thirdline}
\begin{split}
\int_M \left(C | \nabla^{(2)} \varphi| + \right. & \left. C_0 \left( r + \frac{r^2}{\tau}\right) \varphi \right) u |F|^2 \ dV \\
& \leq \int_{B_{\rho_1}(x_1)} \left( C_M e^{-\frac{\rho_1^2}{16 R(t)^2} } + C_0 \left( 1 - \left( \frac{n}{2} - 2 \right) u \log \tau  \right) \right) |F|^2 dV \\
& \leq \left( C_0 + C_M e^{-\frac{\rho_1^2}{16 R(t)^2} } \right) E - C_0 (n - 4) \log \left( \tau(t) \right) \, \Phi(R(t),t) . 
\end{split}
\end{equation}
Substituting (\ref{mainthm:firstline}), (\ref{mainthm:secondline}), and (\ref{mainthm:thirdline}) into (\ref{mainthm:3part}), we obtain
\begin{equation}\label{mainthm:secondtolast}
\begin{split}
\frac{d}{dt} \Phi(R(t), t) & \leq \kappa \left( 1 - \gamma \right) \| F^+(t)\|_{L^\infty} \left( C_2 \Phi(R(t), t)^{1/2} + C_M e^{-\frac{\rho_1^2}{16 R(t)^2}} E^{1/2} \right) \\
& \qquad + \gamma^2 \left( \left( C_0 + C_M e^{-\frac{\rho_1^2}{16 R(t)^2} } \right) E - C_0 ( n - 4) \log \left( \tau(t) \right) \Phi(R(t), t) \right).
\end{split} 
\end{equation}

Following Hamilton \cite{Hamilton1993_2}, p. 133, we use an integrating factor to absorb the $\log$ coefficient. Let
$$\psi(\tau) = \tau \left( 1 - \log \tau \right).$$
Then $\psi'(\tau) = - \log \tau,$ so the function $I(t) = e^{C_0 (n - 4) \gamma^2 \psi(\tau(t)) }$ satisfies
$$\frac{dI}{dt} = C_0(n-4) \gamma^2 \log(\tau(t) ) I(t).$$
Hence, multiplying by $I(t)$ in (\ref{mainthm:secondtolast}), we have
\begin{equation*}
\begin{split}
\frac{d}{dt} I(t) \Phi(R(t), t)  & \leq I(t) \left( C_0 + C_M e^{-\frac{\rho_1^2}{16 R(t)^2} } \right) \gamma^2  E \\
& \qquad + I(t) \kappa (1 - \gamma)  \| F^+(t)\|_{L^\infty} \left( C_2 \Phi(R(t), t) + C_M e^{-\frac{\rho_1^2}{8 R(t)^2}} E \right)^{1/2}.
\end{split}
\end{equation*}
Integrating in time from $t_1$ to $t_2$, and using the fact
$$I(t_1) \leq I(t_2) e^{C_0 \gamma^2 \left(R_1 - R_2 \right)}$$ 
yields the estimate
\begin{equation}\label{overcomplicatedestimate}
\begin{split}
e^{C_0 \gamma^2 (R_2 - R_1) } \Phi(R_2, t_2) & \leq \Phi(R_1, t_1) + \left( C_0 + C_M e^{-\frac{\rho_1^2}{16 R_1^2} } \right) \gamma^2 (t_2 - t_1) E \\
& \qquad + \kappa (1 - \gamma) \int_{t_1}^{t_2}  K(t) \left( C_2 \Phi ( R(t), t ) + C_M e^{-\frac{\rho_1^2}{8 R(t)^2} } E\right)^{1/2} \, dt
\end{split}
\end{equation}
which is equivalent to (\ref{mainestimate}).
\end{proof}

\begin{rmk}\label{rmk:constants} Note from (\ref{c2def}) and (\ref{overcomplicatedestimate}) that if one requires $R_1 < R_0$ sufficiently small, depending on $\rho_1$ and the geometry of $B_{\rho_1}(x_1),$ then $C_2$ may be taken to depend only on the dimension, 
and $C_3$ may be taken arbitrarily small. After rescaling the metric, $C_1$ may also be taken arbitrarily small. In particular, for the case of $\R^n,$ with $\rho_1 = \infty,$ we have 
\begin{equation*}
\begin{split}
\Phi(R_2, t_2) \leq \Phi(R_1, t_1) + C_n \kappa \left(1 - \gamma \right) \int_{t_1}^{t_2}  K(t) \sqrt{ \Phi ( R(t)  , t ) } \, dt.
\end{split}
\end{equation*}
\end{rmk}

\section{Extended $\epsilon$-regularity}


This section uses the extended monotonicity formula of Theorem \ref{thm:mainthm} to derive an $\epsilon$-regularity result, Theorem \ref{thm:modifiedepsilon}. The proof is modeled on Theorem 5.4 of Struwe \cite{struwehm},
which relies implicitly on the following lemma. 


\begin{lemma}[Cf. Struwe \cite{struwehm}, Remark 5.2]\label{lemma:gaussiancomparison} Let $\Phi$ be as in Definition \ref{def:entropy}. Given $\epsilon > 0,$ there exists a constant $C(\epsilon) > 0 $ such that the following holds.




Let $A$ be a connection, $x_1 \in M,$ $x \in B_{R}(x_1),$ and 
\begin{equation}\label{gaussiancomparison:rrequirements}
\frac{R}{2} \leq R_1 \leq \frac{3R}{2}.
\end{equation}
Then
\begin{equation}\label{gaussiancomparison:b}
\Phi_{x_1, \rho_1}(A; R_1, x) \leq C(\epsilon) \Phi_{x_1, \rho_1}(A; R, x_1) + \epsilon \Phi_{x_1, \rho_1}(A; 2R, x_1).
\end{equation}
\end{lemma}
\begin{proof}



Let $Q > 0.$ Note first that for $d(x_1, y) \leq QR,$ we have
\begin{equation}\label{gaussiancomparison:2}
R_1^{4-n} \exp \, \left( - \frac{d(x, y)^2}{4R_1^2} \right) \leq C R^{4 - n} \leq C(Q) R^{4 - n} \exp \, \left( - \frac{d(x_1,y)^2}{4 R^2} \right).
\end{equation}
For $d(x_1, y) \geq QR,$ we write
\begin{equation}\label{gaussiancomparison:3}
\begin{split}
R_1^{4-n} \exp \, \left( - \frac{d(x, y)^2}{4R_1^2} \right) & \leq 2^{n-4} R^{4-n} \exp \left( \frac{d(x_1, y)^2}{16 R^2} - \frac{d(x, y)^2}{4R_1^2} \right) \exp \, - \left(\frac{d(x_1, y)}{4R} \right)^2.
\end{split}
\end{equation}
Note that
\begin{equation*}
\begin{split}
\frac{d(x_1,y)^2}{4R^2} - \frac{d(x, y)^2}{R_1^2} & \leq \frac{d(x_1,y)^2}{ R_1^2} \left( \frac{9}{16} - \left( 1 - \frac{1}{Q} \right)^2 \right) \\
& \leq - c Q^2
\end{split}
\end{equation*}
assuming $Q \geq 5.$ Hence (\ref{gaussiancomparison:3}) becomes
\begin{equation}\label{gaussiancomparison:4}
\begin{split}
R_1^{4-n} \exp \, -\left(\frac{d(x,y)}{2R_1} \right)^2 & \leq C \exp \left( - c Q^2 \right) (2R)^{4-n} \exp \, - \left(\frac{ d( x_1, y) }{4 R} \right)^2
\end{split}
\end{equation}
for $d( x_1, y) \geq QR.$ Combining (\ref{gaussiancomparison:2}) and (\ref{gaussiancomparison:4}), for $Q$ sufficiently large (depending on $\epsilon$), we obtain
\begin{equation}\label{gaussiancomparison:5}
R_1^{4-n} \exp \, -\left(\frac{d(x, y)}{2R_1} \right)^2 \leq C(\epsilon) R^{4 - n} \exp \, - \left( \frac{d(x_1, y)}{2 R} \right)^2 + \epsilon (2R)^{4-n} \exp \, - \left(\frac{d(x_1, y)}{4R} \right)^2.
\end{equation}
The result (\ref{gaussiancomparison:b}) follows by integrating (\ref{gaussiancomparison:5}) against $|F_A(y)|^2 \varphi_{x_1, \rho_1}(y).$
\end{proof}

\begin{thm}
\label{thm:modifiedepsilon}
Let $E,E_0 > 0,$ and $x_1 \in M$ with $\inj(M, x_1) \geq \rho_1 > 0.$ There exists a constant $\epsilon_0 > 0,$ depending only on $E_0$ and $n,$ as well as $R_0 > 0,$ depending on $E,$ $\min \LB \rho_1, 1 \RB,$ and the geometry of $B_{\rho_1}(x_1) \subset M,$ as follows. 

Let $A(t)$ be a $\gothk$-compatible solution of (YM) on $M \times \LB 0, T \right),$ and define $K(t)$ by (\ref{kdefinition}). Choose 
$$0 < R < R_0, \qquad 0 \leq t_1 < T, \qquad t_1 + R^2 \leq t_2 \leq T$$
and let
$$\gamma = \frac{R}{\sqrt{t_2 - t_1}}.$$
Assume
\begin{equation}\label{modifiedepsilon:eassumption}
\sup_{t_1 \leq t \leq t_2} \int_{B_{\rho_1}(x_1)} | F(t) |^2 \, dV \leq E
\end{equation}
\begin{equation}\label{modifiedepsilon:e0assumption}
\Phi_{x_1, \rho_1}(2R, x_1, t_1) \leq E_0
\end{equation}
and
\begin{equation}\label{modifiedepsilon:assumption}
\Phi_{x_1, \rho_1}(R,x_1, t_1) + \kappa \left( 1 - \gamma \right) \int_{t_1}^{t_2} K(t) \, dt < \epsilon_0.
\end{equation}
Then
\begin{equation}\label{modifiedepsilon:estimate}
\sup_{B_{R/2}(x_1) \times \LB t_1 + \frac{R^2}{2} , t_2 \right) } | F(x,t) | \leq \frac{C_n}{R^2}.
\end{equation}
\end{thm}


\begin{proof} 




Let $\epsilon_1 > 0.$ We first claim that it is possible to choose $R_0, \epsilon_0 > 0$ sufficiently small, so that for all
\begin{equation}\label{modifiedepsilon:0}
0 < \sigma \leq R, \qquad x \in B_{R}(x_1), \qquad t_1 + \frac{R^2}{4} \leq t < t_2
\end{equation}
we have
\begin{equation}\label{modifiedepsilon:1}
\Phi_{x_1, \rho_1}(\sigma, x, t) \leq \epsilon_1. 
\end{equation}

Fix $\sigma, x,$ and $t$ satisfying (\ref{modifiedepsilon:0}). To prove (\ref{modifiedepsilon:1}), let
\begin{equation*}
\begin{split}
R_2 = \sigma, \qquad \gamma_0 = \min \LB \frac{R}{\sqrt{t - t_1}}, 1 \RB, \qquad R_1 = \sqrt{ \sigma^2 + \gamma_0^2 \left( t - t_1 \right) }.
\end{split}
\end{equation*}
Note that $R / 2 \leq R_1 \leq 3R / 2,$ as required by (\ref{gaussiancomparison:rrequirements}). 
Letting $R(s) = \sqrt{\sigma^2 + \gamma_0^2 \left( t - s \right)},$ 
for any $t_1 \leq t' \leq t,$ the monotonicity formula (\ref{mainestimate}) reads
\begin{equation}\label{modifiedepsilon:2}
\begin{split}
& \Phi_{x_1, \rho_1}(R(t), x, t') \leq e^{C_0 \gamma_0^2 \left( R_1 - R_2 \right)} \Phi_{x_1, \rho_1}(R_1, x, t_1) + C_1 R^2 E \\
& \qquad \qquad \qquad \qquad \quad + \kappa (1 - \gamma_0) \int_{t_1}^{t'} K(s) \left( C_2 \Phi_{x_1, \rho_1} ( R(s), x, s ) + C_3 E \right)^{1/2} \, ds .
\end{split}
\end{equation}
According to Remark \ref{rmk:constants}, by choosing $R_0$ sufficiently small, we may assume that $C_2$ depends only on $n,$ and
\begin{equation}\label{modifiedepsilon:2.5}
C_0 R \leq \log 2, \qquad C_1 R^2 E \leq \frac{\epsilon_1}{6}, \qquad C_3 \leq \frac{1}{E}.
\end{equation}
From Lemma \ref{lemma:gaussiancomparison}, given $\epsilon > 0,$ we also have
\begin{equation}\label{modifiedepsilon:3}
\begin{split}
\Phi_{x_1, \rho_1} (R_1, x, t_1) & \leq C(\epsilon) \Phi_{x_1, \rho_1} (R, x_1, t_1) + \epsilon \Phi_{x_1, \rho_1} (2R, x_1, t_1) \\
& \leq C(\epsilon) \epsilon_0 + \epsilon E_0.
\end{split}
\end{equation}
Inserting (\ref{modifiedepsilon:eassumption}-\ref{modifiedepsilon:assumption}) and (\ref{modifiedepsilon:2.5}-\ref{modifiedepsilon:3}) into (\ref{modifiedepsilon:2}) yields
\begin{equation}\label{modifiedepsilon:4}
\begin{split}
& \Phi_{x_1, \rho_1}(R(t), x, t') \\
& \qquad \leq 2 \left( C(\epsilon) \epsilon_0 + \epsilon E_0 \right) + \frac{\epsilon_1}{6}
+ \kappa (1 - \gamma_0) \int_{t_1}^{t'}  K(s) \left(C \Phi_{x_1, \rho_1}( R(s), x  , s ) + 1 \right)^{1/2} \, ds \\
& \qquad \leq \left( C(\epsilon) + C \left( \sup_{t_1 \leq s \leq t'} \Phi_{x_1, \rho_1} (R(s), x, s) + 1 \right)^{1/2} \right) \epsilon_0 + 2 \epsilon E_0 + \frac{\epsilon_1}{6}
\end{split}
\end{equation}
where we have used that $1 - \gamma_0 \leq 1 - \gamma.$

We claim that (\ref{modifiedepsilon:4}) implies (\ref{modifiedepsilon:1}). Indeed, let
$$P = \sup_{t_1 \leq s \leq t} \Phi_{x_1, \rho_1}(R(s), x, s).$$
Then (\ref{modifiedepsilon:4}) reads 
\begin{equation*}\label{modifiedepsilon:6}
\begin{split}
P \leq \left( C(\epsilon) + C \left( \sqrt{P} + 1 \right) \right) \epsilon_0 + 2 \epsilon E_0 + \frac{\epsilon_1}{6}
\end{split}
\end{equation*}
which may be rewritten
\begin{equation}\label{modifiedepsilon:7}
\begin{split}
\sqrt{P} \left( \sqrt{P} - C \epsilon_0 \right) \leq \left( C(\epsilon) + C \right) \epsilon_0 + 2 \epsilon E_0 + \frac{\epsilon_1}{6}.
\end{split}
\end{equation}
Assume that $\left(2 C \epsilon_0 \right)^2 \leq \epsilon_1.$ Then, if $\sqrt{P} \leq 2 C \epsilon_0,$ (\ref{modifiedepsilon:1}) is proved. If $\sqrt{P} \geq 2 C \epsilon_0,$ we have
\begin{equation}
P \leq 2 \sqrt{P} \left( \sqrt{P} - C \epsilon_0 \right)
\end{equation}
and (\ref{modifiedepsilon:7}) reads
\begin{equation}\label{modifiedepsilon:7.5}
\begin{split}
P & \leq \left( C(\epsilon) + C \right) \epsilon_0 + 4 \epsilon E_0 + \frac{\epsilon_1}{3}.
\end{split}
\end{equation}
We now choose
\begin{equation*}
\epsilon = \frac{\epsilon_1}{12 E_0},  \qquad \epsilon_0 \leq \frac{\epsilon_1}{3 \left(C(\epsilon) + C \right)}.
\end{equation*}
Then (\ref{modifiedepsilon:7.5}) yields
\begin{equation*}
P \leq 3 \frac{\epsilon_1}{3} = \epsilon_1
\end{equation*}
which proves the claim concerning (\ref{modifiedepsilon:1}). 

Directly from (\ref{modifiedepsilon:1}) and Definition \ref{def:entropy}, for any
$$0 < \sigma < R, \qquad x \in B_{R}(x_1), \qquad t_1 + \frac{R^2}{4} \leq t < t_2$$
we now have
\begin{equation}\label{modifiedepsilon:8}
\begin{split}
\sigma^{4-n} \int_{B_\sigma (x)}  |F(y,t)|^2 \, dV_y & \leq C \Phi_{x_1, \rho_1}(\sigma, x, t) \\
& \leq \epsilon_1
\end{split}
\end{equation}
where we have written $\epsilon_1$ in place of $C \epsilon_1$ (since it was arbitrary).

The remainder of the proof now follows the standard argument. Since $R < R_0,$ we may assume that compactly supported functions on $B_R $ obey the Sobolev inequality with a constant depending only on dimension. 
We rescale parabolically so that $R = 1,$ preserving (\ref{modifiedepsilon:8}) as well as the validity of the conclusion (\ref{modifiedepsilon:estimate}), and write $x_1 = 0$ in geodesic coordinates.

Denote the parabolic cylinders
$$P_r(x,t) = B_r(x) \times \LB t - r^2, t \right).$$
Fix $\tau$ with $t_1 + \frac{3}{4} \leq \tau < t_2,$ and write
\begin{align}\label{modifiedepsilon:8.3}
\nonumber P_r & = P_r (0, \tau) \\
\nonumber w(x,t) & = |F(x,t)|^2 \\
e(r) & = \left( \frac{1}{\sqrt{2}} - r \right)^4 \sup_{P_r } w(x,t).
\end{align}
Let $e_0, r_0$ be such that
\begin{equation}\label{modifiedepsilon:8.6}
e_0 = e(r_0) = \sup_{0 \leq r \leq \frac{1}{\sqrt{2}} } e(r).
\end{equation}
Choose $(x_0, t_0) \in \bar{P}_{r_0}$ such that $w(x_0, t_0) = \sup_{P_{r_0}} w.$ Letting
$\sigma_0 = \frac{1}{2}\left( \frac{1}{\sqrt{2}}  - r_0 \right), $ we have
$$P_{\sigma_0}(x_0, t_0) \Subset P_{1/\sqrt{2}}$$
and
\begin{samepage}
\begin{equation}\label{modifiedepsilon:9}
\sigma_0^4 \sup_{P_{\sigma_0}(x_0, t_0)} w \leq 16 e_0
\end{equation}
by choice of $r_0.$
\end{samepage}

Assume first that $e_0 > 1.$ Letting
$$\sigma_1 = \frac{\sigma_0 }{ 2 e_0 ^{1/4} }$$
we may perform the further rescaling
$$w_1(x,t) = \sigma_1^4 w\left( x_0 + \sigma_1 x , t_0 + \sigma_1^2 (t - t_0) \right)$$
to obtain a function $w_1(x,t)$ defined on $P = P_1(0, 0).$ 
Then (\ref{modifiedepsilon:9}) reads
\begin{equation*}
\sup_{P} w_1(x,t) \leq 1
\end{equation*}
and the differential inequality (\ref{ymcurvatureweitzwithnorms}) becomes
\begin{equation*}
\left( \partial _t + \Delta \right) w \leq C \left( w + w^{3/2} \right) \leq C w.
\end{equation*}
But then Moser's Harnack inequality (\textit{cf.} Lemma 2.2 of \cite{instantons}) gives
\begin{equation*}
1 = w_1(0,0) \leq C \int \!\!\!\!\! \int_{P} w_1(x,t) \, dV dt \leq C \epsilon_1
\end{equation*}
owing to (\ref{modifiedepsilon:8}). For $\epsilon_1 < 1/C,$ this is a contradiction.

We therefore conclude that $e_0 \leq 1.$
Directly from (\ref{modifiedepsilon:8.3}) and (\ref{modifiedepsilon:8.6}), we now have
$$\sup_{P_{1/2}} u \leq C e_0 \leq C.$$
Since $\tau$ was arbitrary (within the relevant time interval), this implies the desired estimate (\ref{modifiedepsilon:estimate}).
\end{proof}



\begin{rmk} 
The next result, Theorem \ref{thm:epsilon0independent}, is
the analogue of Theorem 5.1 of Struwe \cite{struwehm}, which we include here only for the sake of completeness. The last result, 
analogous to Theorem 5.3 of \cite{struwehm}, achieves the simplest version of the $\epsilon$-regularity theorem (at the price of letting $\epsilon_0$ depend on $M$). 
\end{rmk}

\begin{thm}\label{thm:epsilon0independent}
There exists a constant $\epsilon_0 > 0,$ depending only on $n,$ as well as $R_0 > 0,$ depending on $E,$ $\min \LB \rho_1, 1 \RB,$ and the geometry of $B_{\rho_1}(x_1) \subset M,$ as follows. 

With otherwise the same setup as Theorem \ref{thm:modifiedepsilon}, omit (\ref{modifiedepsilon:e0assumption}) and assume
\begin{equation*}
\sup_{t_1 \leq t \leq t_2} \int_{B_{\rho_1}(x_1)} | F(t) |^2 \, dV \leq E
\end{equation*}
and
\begin{equation*}
\Phi_{x_1, \rho_1}(R, x_1, t_1) + \kappa \left( 1 - \gamma \right) \int_{t_1}^{t_2} K(t) \, dt < \epsilon_0.
\end{equation*}
Then, there exists a constant $\delta > 0,$ depending on $\min \LB R, 1 \RB, E,$ and $n,$ such that
\begin{equation*}
\sup_{B_{\delta R}(x_1) \times \LB t_1 + \left( 1 - \delta^2 \right) R^2 , t_2 \right) } | F(x,t)| \leq \frac{C_n}{\left( \delta R\right)^2}.
\end{equation*}
\end{thm}
\begin{proof} This follows by the same proof as Theorem \ref{thm:modifiedepsilon}, using the comparison estimate (5.2) of Struwe \cite{struwehm} in place of Lemma \ref{lemma:gaussiancomparison}.
\end{proof}

\begin{cor}\label{cor:modifiedepsiloncor}
There exists $\epsilon_0 > 0,$ depending on $E,$ $\min \LB \rho_1, 1 \RB,$ and the geometry of $B_{\rho_1}(x_1) \subset M,$ as follows.

With otherwise the same setup as Theorem \ref{thm:modifiedepsilon}, omit (\ref{modifiedepsilon:e0assumption}) and assume 
$$0 < R < \epsilon_0, \qquad \rho_1^2 \leq t_1 < T, \qquad t_1 + R^2 \leq t_2  \leq T.$$
If
\begin{equation}\label{modifiedepsiloncor:Eassumption}
\sup_{ 0 \leq t < T} \int_{B_{\rho_1}(x_1)} | F(t) |^2 \, dV \leq E
\end{equation}
and
\begin{equation}\label{modifiedepsiloncor:epsilonassumption}
\Phi_{x_1, \rho_1}(R,x_1, t_1) + \kappa \left( 1 - \gamma \right) \int_{t_1}^{t_2} K(t) \, dt < \epsilon_0
\end{equation}
then
\begin{equation}
\sup_{B_{R/2}(x_1) \times \LB t_1 + \frac{R^2}{2} , t_2 \right) } | F(x,t) | \leq \frac{C_n}{R^2}.
\end{equation}
\end{cor}

\begin{proof} Because $t_1 \geq \rho_1^2,$ an estimate of the form (\ref{modifiedepsilon:e0assumption}) now follows from (\ref{modifiedepsiloncor:Eassumption}) by an extra application of Theorem \ref{thm:mainthm}, with $\gamma = 1,$ $R_1 = \rho_1,$ and $R_2 = 2R.$ The resulting $E_0$ will now depend on $E,$ $\rho_1,$ and the geometry of $M.$ Applying Theorem \ref{thm:modifiedepsilon}, we obtain an $\epsilon_0 > 0$ which now depends on all of the above, and may assume $\epsilon_0 \leq R_0$ to eliminate the extra constant.
\end{proof}


\section{Blowup criteria on special-holonomy manifolds}\label{sec:blowupcriteria}

This section derives several corollaries of Theorem \ref{thm:modifiedepsilon}, including our main result, which follows.

\begin{thm}\label{thm:mainblowupcrit}
	Let $M$ be a Riemannian manifold (without boundary) which carries a torsion-free $N(G)$-structure as in \S \ref{sec:splitting}. Let $A(t)$ be a $\gothk$-compatible smooth solution of (YM) on $M \times \LB 0, T \right),$ with $T \leq \infty.$ 
	
Assume that for each compactly contained open set $\Omega \Subset M,$ we have
\begin{equation}\label{mainblowupcrit:Eassn}
	\sup_{ 0 \leq t < T } \int_{\Omega} | F(t) |^2 \, dV  < \infty
	\end{equation}
 and
\begin{equation}\label{f+bounded}
\int_0^T  \| F^+ (t) \|_{L^\infty(\Omega)} \, dt < \infty.
\end{equation}
Then, the curvature $|F(t)|$ remains locally bounded as $t \nearrow T,$ \textit{i.e.} 
\begin{equation}\label{mainblowupcrit:supbound}
\limsup_{\stackrel{t \nearrow T}{x \in \Omega}} |F(x,t)| < \infty
\end{equation}
for each $\Omega \Subset M.$
Moreover, if $T < \infty$ and $M$ is compact, then the flow extends smoothly past $T.$

A similar result holds with $F^+$ replaced by $\sum_{\alpha \neq \beta} F^\alpha,$ for any $\beta$ such that the eigenvalue $\lambda_\beta$ of (\ref{staroperator}) is nonzero.
\end{thm}

\begin{proof} 

Let $x \in M,$ and choose $0 < \rho_1 < \min \LB \inj(M,x), \sqrt{T} \RB.$ By (\ref{mainblowupcrit:Eassn}), we have
\begin{equation}
\sup_{ 0 \leq t < T } \int_{B_{\rho_1}(x)} | F(t) |^2 \, dV  \leq E
\end{equation}
for some $E > 0;$ hence (\ref{modifiedepsiloncor:Eassumption}) is satisfied.
By (\ref{f+bounded}), we may choose $\rho_1^2 \leq t_1 < T$ sufficiently close that
\begin{equation}\label{mainblowupcrit:1}
\int_{t_1}^T  \| F^+ (\cdot, t) \|_{L^\infty(M)} \, dt < \frac{\epsilon_0}{2\kappa}.
\end{equation}
Since $A(t_1)$ is smooth, we may also choose $0 < R < \sqrt{T - t_1}$ sufficiently small that
\begin{equation}\label{mainblowupcrit:2}
\Phi_{x, \rho_1}(R, x, t_1) < \frac{\epsilon_0}{2}.
\end{equation}
Combining (\ref{mainblowupcrit:1}) and (\ref{mainblowupcrit:2}) yields (\ref{modifiedepsiloncor:epsilonassumption}), with $\gamma = 0$ and $t_2 = T.$ 
By Corollary \ref{cor:modifiedepsiloncor}, we conclude that the full curvature $|F|$ remains uniformly bounded in a neighborhood of $x$ as $t \nearrow T.$ 
Since $x$ was arbitrary, this is equivalent to (\ref{mainblowupcrit:supbound}).

Given (\ref{mainblowupcrit:supbound}), and assuming $T < \infty,$ Lemma 2.4 
of \cite{waldronuhlenbeck} implies that $A(t)$ converges in $C^\infty_{loc}(M)$ as $t \nearrow T.$ Hence, if $M$ is compact, we may restart the flow (briefly) at time $T.$
\end{proof}

\begin{proof}[Proof of Theorem \ref{thm:Main_1_Introduction}] According to \S \ref{ss:g2splitting}-\ref{ss:spin7splitting}, in the case of $\rG_2$ or $\rSpin(7)$ holonomy, we have $F^+ = F^7$ and the compatibility condition is trivial. By assumption, $M$ is compact and $T < \infty;$ hence (\ref{mainblowupcrit:Eassn}) follows from the global energy identity (\ref{eq:Energy_Identity}), and (\ref{f+bounded}) is implied by (\ref{f7bounded}). Theorem \ref{thm:Main_1_Introduction} therefore follows from Theorem \ref{thm:mainblowupcrit}.
\end{proof}

\begin{cor}\label{cor:bothblowup} 
Assume that the operator (\ref{staroperator}) is invertible. Let $A(t)$ solve (YM) as above, with $T < \infty.$ If, for any open set $\Omega \Subset M,$ $\lim_{t \nearrow T} A(x,t)$ fails to exist in $C^\infty(\Omega),$ then
$$\limsup_{t \nearrow T } \| F^\alpha (t) \|_{L^\infty(M)} = \infty = \limsup_{t \nearrow T } \| F^\beta (t) \|_{L^\infty(M)} $$
for at least two components $\alpha \neq \beta$ in the eigenspace decomposition (\ref{falphadecomp}). 
\end{cor}
\begin{proof} This follows from the contrapositive of the Theorem \ref{thm:mainblowupcrit}.
\end{proof}

\begin{cor}\label{cor:Exponential_Blow_Up}
Let $A(t)$ be as above, and assume that $M$ is compact. Define
$$K(t) = \sup_{x \in M} |F^+(x,t)|, \qquad L(t) = \sup_{x \in M} |F(x,t)|, \qquad E = \int_M |F(0)|^2 \, dV.$$
For $0 \leq t_1 \leq t_2 < T,$ 
either $L(t_2) \leq R_0^{-2}$, or
\begin{equation}\label{eq:Exponential_Blow_Up}
\log L(t_2) \leq \log L(t_1) +  C_n \max \LB 1, \kappa \int_{t_1}^{t_2} K(t) \, dt \RB.
\end{equation}
Here $R_0 > 0$ is a constant depending on $E$ and $M.$
\end{cor}
\begin{proof} Let $E_0 = 1,$ so that the constant $\epsilon_0$ of Theorem \ref{thm:modifiedepsilon} depends only on the dimension. Since $M$ has bounded geometry, we may take $\rho_1 = \min \LB 1, \inj(M) \RB,$ and $R_0$ independent of $x_1$ in Theorem \ref{thm:modifiedepsilon}.

We will prove the following equivalent statement: if $L(t_1) \geq R_0^{-2}$ and 
\begin{equation}\label{exponentialblowup:1}
\int_{t_1}^{t_2} K(s) \, ds < \frac{\epsilon_0}{2 \kappa}
\end{equation}
then
\begin{equation}\label{exponentialblowup:2}
L(t_2) < C_n L(t_1).
\end{equation}
First, recall that by a standard barrier argument applied to (\ref{ymcurvatureweitzwithnorms}), for $t \leq t_1 + c_n L(t_1)^{-2}$ and $L(t_1) \geq R_0^{-2},$ there holds
\begin{equation}
L(t) \leq 2 L(t_1).
\end{equation}
Hence, it suffices to assume 
\begin{equation}\label{t2assumption}
t_2 \geq t_1 + c_n L(t_1)^{-2}.
\end{equation}

Given $x \in M,$ note that
\begin{equation}\label{phiestimate}
\begin{split}
\Phi_{x, \rho_1}(R, x, t ) & = \int \!\! R^{4} |F(y,t)|^2 G_{x}(y,t) \varphi_{x, \rho_1}(y) \, dV_y \\
& \leq R^4 L(t)^2.
\end{split}
\end{equation}
Let
\begin{equation*}
R = \frac{\min \LB c_n^{1/2}, \left( \dfrac{\epsilon_0}{ 2} \right)^{1/4} \RB}{L(t_1)^{1/2}}.
\end{equation*}
We have $t_1 + R^2 \leq t_2$ from (\ref{t2assumption}), as well as
\begin{equation*}
\begin{split}
\Phi_{x, \rho_1}(2R, x, t_1) & \leq 1 \\
\Phi_{x, \rho_1}(R, x, t_1) & \leq \frac{\epsilon_0}{2}
\end{split}
\end{equation*}
from (\ref{phiestimate}). Combined with (\ref{exponentialblowup:1}), this yields (\ref{modifiedepsilon:eassumption}-\ref{modifiedepsilon:e0assumption}). Applying Theorem \ref{thm:modifiedepsilon}, we have
\begin{equation*}
|F(x,t_2)| \leq \frac{C_n}{R^2} \leq C_n L(t_1).
\end{equation*}
Since $x \in M$ was arbitrary, this implies (\ref{exponentialblowup:2}).

The statement (\ref{exponentialblowup:1}-\ref{exponentialblowup:2}) implies (\ref{eq:Exponential_Blow_Up}) by subdividing the interval $\LB t_1, t_2 \RB.$
\end{proof}


\begin{cor}[Donaldson \cite{Donaldson1985, donaldsoninfinitedeterminants}]\label{cor:Long_Time_Kahler}
	Let $A_0$ be a connection on a Hermitian vector bundle $E$ over a compact K\"ahler manifold $M,$ with curvature of type $(1,1).$ Then, the Yang-Mills flow $A(t)$ with $A(0)=A_0$ exists for all time, with curvature $F(t)$ blowing up at most exponentially as $t \rightarrow \infty$.
\end{cor}
\begin{proof}
	
By Proposition \ref{prop:Orthogonal_Curvature_Vanishing}, the solution $A(t)$ remains $\mathfrak{u}(n)$-compatible (\textit{i.e.} has curvature of type $(1,1)$) for as long as it exists. Taking an inner product with $F^\omega = \left(\Lambda_\omega F \right) \omega$ in the evolution equation (\ref{kahlerfplus}), we obtain 
	\begin{equation}\nonumber
	\begin{split}
	\frac{1}{2} \del_t |F^{\omega}|^2 & = - \langle F^{\omega}, \nabla^* \nabla F^{\omega}\rangle.
	\end{split}
	\end{equation}
Combining with the identity
	$$- \langle F^{\omega} , \nabla^* \nabla F^{\omega} \rangle = - \frac{1}{2} \Delta |F^{\omega}|^2 - |\nabla F^{\omega}|^2$$
	yields
	\begin{equation}\nonumber
	\begin{split}
	(\del_t + \Delta )|F^{\omega}|^2 & = - 2 | \nabla F^{\omega} |^2  \leq 0 .
	\end{split}
	\end{equation}	
By the maximum principle, $|F^{\omega}(t)| = |F^+(t)|$ remains uniformly bounded, and long-time existence follows from Theorem \ref{thm:mainblowupcrit}. By Corollary \ref{cor:Exponential_Blow_Up}, the full curvature $|F(t)|$ blows up at most exponentially as $t \rightarrow \infty$.
\end{proof}

\begin{cor}\label{cor:quatkahler} Let $A_0$ be a pseudo-holomorphic connection (see \S \ref{ss:quatkahler}) on a vector bundle over a compact quaternion-K\"ahler manifold $M.$ There exists $T > 0$ and a smooth solution $A(t)$ of (YM) on $M\times \LB 0, T \right),$ with $A(0) = A_0$ and $A(t)$ pseudo-holomorphic for $0 \leq t < T.$ If $T < \infty$ is maximal, then
$$\limsup_{\stackrel{x\in M}{ t \nearrow T} } | F^{\Omega}(x,t) | = \infty.$$
\end{cor}
\begin{proof} This follows from the discussion in \S \ref{ss:quatkahler}, Proposition \ref{prop:Orthogonal_Curvature_Vanishing}, and Corollary \ref{cor:bothblowup}.
\end{proof}

\section{Infinite-time singular set}\label{sec:calibratedness}

This brief section contains the detailed version of the second theorem of the introduction, Theorem \ref{thm:calibrated}. The proof is based on Theorem \ref{thm:mainblowupcrit}, the results of \cite{waldronuhlenbeck}, and the following well-known lemma.

\begin{lemma}[Cf. Tian \cite{tiancalibrated}, Corollary 4.2.2]\label{lemma:n-4calibrated} Let $\Psi$ be a linear $(n-4)$-calibration on an oriented Euclidean vector space $V,$ and choose an $(n-4)$-plane $U \subset V.$ Let $B$ be a connection on a bundle over $V$ which is the product of a flat connection on $U$ with a non-flat connection on $U^\perp \cong \R^4.$

If $B$ is a $\Psi$-instanton, then $\left. B \right|_{U^{\perp}}$ is anti-self-dual and $U$ is a calibrated plane. If, also, $\Psi$ is preserved by an $\rSU(2)$-subgroup as in Lemma \ref{lemma:eigen}, then the converse holds. 
\end{lemma}
\begin{proof} 
We write
\begin{equation}\label{n-4calibrated:0}
\begin{split}
& \Psi = \alpha \mathrm{Vol} _{U} + \Psi' \\
& \Psi' \in \oplus_{i = 1}^4 \Lambda^{n - 4 - i} U \otimes \Lambda^i U^\perp.
\end{split}
\end{equation}
The orientation on $U$ may be chosen so that $\alpha \in [0,1]$ in (\ref{n-4calibrated:0}).  

If $B$ is a $\Psi$-instanton, then $F = F_B$ satisfies
\begin{equation}\label{n-4calibrated:1}
\begin{split}
F & = - \alpha \ast (F \wedge dV_{U}) - * (F \wedge \Psi') \\
& = - \alpha \ast_{U^{\perp}} F - * (F \wedge \Psi').
\end{split}
\end{equation}
But we have
\begin{equation*}
F \wedge \Psi' \in \Lambda^{n - 5} U \otimes \Lambda^{3} U^\perp \oplus \Lambda^{n - 6} U \otimes \Lambda^{4} U^\perp 
\end{equation*}
and
\begin{equation*}
* \left( F \wedge \Psi' \right) \in U \otimes U^\perp \oplus \Lambda^2 U.
\end{equation*}
The latter must vanish, since the other terms of (\ref{n-4calibrated:1}) lie in $\Lambda^2 U^\perp.$ Then (\ref{n-4calibrated:1}) reduces to
\begin{equation*}
F = - \alpha \ast_{U^{\perp}} F .
\end{equation*}
This yields
\begin{equation*}
\begin{split}
F & = - \alpha \ast_{U^{\perp}} F = (- \alpha)^2 \ast_{U^{\perp}}^2 F = (- \alpha)^2 F.
\end{split}
\end{equation*}
Since $F \not\equiv 0$ we must have $\alpha = 1,$ and both claims follow.

The converse 
follows from the proof of Lemma \ref{lemma:eigen}.
\end{proof}

\begin{thm}\label{thm:calibrated}
Let $A(t)$ be a $\gothk$-compatible smooth solution of (YM) on $M \times \LB 0 , T \right),$ with $T$ maximal, over a compact manifold $M$ admitting a torsion-free $N(\rG)$-structure as in \S \ref{sec:splitting}. Assume that
\begin{equation}\label{f+supbound}
\sup_{\stackrel{0 \leq t < T}{x \in M}} |F^+(x,t)| < \infty.
\end{equation}
Then $T = \infty,$ and 
for any sequence $t_i \nearrow \infty,$ we may pass to a subsequence of $\{t_i\},$ again indexed by $i,$ as follows.

The set
\begin{equation}
\Sigma = \lbrace x \in M \ | \ \liminf_{R \searrow 0} \liminf_{i \to \infty} \Phi \left(R, x, t_i - R^2 \right) \geq \epsilon_0 \rbrace
\end{equation}
is a closed $(n-4)$-rectifiable set of finite $\mathcal{H}^{n-4}$-measure, which is calibrated by $\Psi.$ There exists an Uhlenbeck limit $A_\infty,$ which is a Yang-Mills connection on a vector bundle $E_\infty \to M \setminus \Sigma,$ together with bundle maps $u_i : E \to E_\infty$ (defined on an exhaustion of $M \setminus \Sigma$), such that
\begin{equation}\label{calibrated:uhlenbecklimit}
u_iA(x, t_i + t) \to \underline{A}_\infty \mbox{ in } C^\infty_{loc}\left( \left(M \setminus \Sigma \right) \times \R \right).
\end{equation}
 Here $\underline{A}_\infty$ is the constant solution of (YM) on $E_\infty$ identically equal to $A_\infty.$
 
Moreover, for $\mathcal{H}^{n-4}$-almost-every $x \in \Sigma,$ there exist $\lambda_i \searrow 0,$  $x_i \to x,$ and $\tau_i \to \infty,$ with $t_i - \tau_i \nearrow 0,$ such that the blowup sequence
\begin{equation}\label{calibrated:blowupsequence}
\lambda_i A_i \left(x_i + \lambda_i y, \tau_i + \lambda_i^2 t  \right)
\end{equation}
converges smoothly modulo gauge to the constant product on $T_x M$ of a nonzero finite-energy anti-self-dual connection on $(T_x \Sigma)^{\perp} \cong \R^4$ with a flat connection on $T_x \Sigma \cong \R^{n-4}.$
\end{thm}
\begin{proof} Proposition \ref{prop:Orthogonal_Curvature_Vanishing} implies that $A(t)$ remains $\gothk$-compatible (per Definition \ref{def:Compatible_Connection}) on $\LB 0, T \right).$ Because (\ref{f+supbound}) implies (\ref{f+bounded}) at finite time, and $T$ is maximal, we conclude from Theorem \ref{thm:mainblowupcrit} that $T = \infty.$

The existence of the Uhlenbeck limit $A_\infty,$ satisfying (\ref{calibrated:uhlenbecklimit}), and the rectifiability of $\Sigma,$ follow from Corollary 1.3 of \cite{waldronuhlenbeck}. Theorem 4.1 of \cite{waldronuhlenbeck}, applied to the sequence of solutions $A_i(t) = A(t_i + t),$ gives the required blowup sequence (\ref{calibrated:blowupsequence}) for $\mathcal{H}^{n-4}$-a.e. $x \in \Sigma,$ in which the blowup limit $B_x$ reduces to the product of a non-flat Yang-Mills connection on $T_x \Sigma^{\perp}$ with a flat connection on $T_x \Sigma.$

It follows from (\ref{f+supbound}) that $B_x$ is a $\Psi(x)$-instanton on $T_x M \cong \R^n.$ 
By Lemma \ref{lemma:n-4calibrated}, this forces $T_x \Sigma \cong \R^{n-4}$ to be a $\Psi(x)$-calibrated $(n-4)$-plane, with $B_x$ anti-self-dual on $T_x \Sigma^{\perp} \cong \R^4.$ Since $x \in \Sigma$ was arbitrary (up to $\mathcal{H}^{n-4}$-measure zero), we are done.
\end{proof}

%

\begin{rmk} If we replace (\ref{f+supbound}) by the weaker assumption
\begin{equation*}
\lim_{r \searrow 0} \limsup_{t \nearrow T} \| F^+(t) \|_{L^2(\Sigma_{r})} = 0
\end{equation*}
and allow $T \leq \infty,$ 
the same conclusions may be drawn regarding $\Sigma.$
Note, however, that for $T < \infty,$ $(n-4)$-rectifiability and calibratedness of $\Sigma$ are trivial if Conjecture 1.5 of \cite{waldronuhlenbeck} holds.
\end{rmk}

\section{Holonomy reductions}\label{sec:reductions}

This section works out the consequences of Theorem \ref{thm:Main_1_Introduction} when the holonomy of $M$ reduces to a proper subgroup of $\rSpin(7).$ The results follow by identifying the $7$-component of the curvature in each case. 

\subsection{From $\rSpin(7)$ to $\rSU(2)$} 

Let $A$ be a connection on a vector bundle $E$ over a compact hyperk\"ahler $4$-manifold $X$ (\textit{i.e.} a $K3$ surface or $T^4$). Let $\Phi_i \in \Omega^0(X, \mathfrak{so}(E))$, for $i=0,1,2,3$.
Pulling back these objects to $M = X \times T^4,$ we obtain a connection $\mathbb{A}= A + \sum_{i=0}^3 \Phi_i \otimes d\theta_i$ over $M,$ with curvature 
\begin{equation*}
F_{\mathbb{A}} = F_A + \sum_{i=0}^3 D_A \Phi_i \wedge d \theta_i + \sum_{i<j} [\Phi_i,\Phi_j] d\theta_i \wedge d\theta_j.
\end{equation*}
Here $\theta_i$ are periodic coordinates on $T^4.$
The Yang-Mills energy of $\mathbb{A}$ amounts to the following functional of $A$ and $\Phi_i$:
\begin{equation*}\label{eq:Functional_K3}
\mathcal{E}(A, \{\Phi_i\}) = \frac{1}{2} \int_X |F_A|^2 + \sum_{i=0}^3 | D_A \Phi_i |^2 + \sum_{i<j} |[\Phi_i,\Phi_j]|^2.
\end{equation*}
The Yang-Mills flow on $M$ is equivalent to the gradient flow of $\mathcal{E}$ on $X,$ given by
\begin{equation}\label{eq:Evolution_K3}
\begin{split}
\qquad \qquad \frac{\partial A}{\partial t} & = - D^*F - \sum_{i=0}^3 [\Phi_i , D \Phi_i]  \qquad \\
\frac{\partial \Phi_i}{\partial t} & = - \Delta \Phi_i - \sum_{j=0}^3 [\Phi_j , [\Phi_i , \Phi_j]]  \qquad \left( i = 0,1,2,3 \right).
\end{split}
\end{equation}
Note that the fields $\Phi_i$ remain bounded for as long as the flow is defined, by the maximum principle:
\begin{equation*}\label{eq:Phi_i_Bounded}
\begin{split}
\left( \frac{\partial}{\partial t} + \Delta \right) \frac{ | \Phi_i |^2}{2} & = \langle \Phi_i , \frac{\partial \Phi_i}{\partial t} \rangle + \langle \Phi_i , \Delta_A \Phi_i \rangle - |D \Phi_i |^2 \\
& =  - | D \Phi_i |^2 -  \sum_{j=0}^3 |[\Phi_i , \Phi_j]|^2 \leq 0.
\end{split}
\end{equation*}

Fix a global covariant-constant frame $\{\omega_i\}$ for the self-dual 2-forms $\Omega^2_{+}(X),$ with $|\omega_i|^2 = 2.$ Let $\tau_i = d\theta^{0i}+\theta^{jk}$, 
for $(i,j,k)$ a cyclic permutation of $(1,2,3),$ 
and equip $M$ with the $\rSpin(7)$-structure defined by
\begin{equation*}
\Theta = \text{dV}_X - \sum_{i=1}^3 \omega_i \wedge \tau_i + \text{dV}_{T^4}.
\end{equation*}

\begin{lemma}
	Denote the curvature of $A$ as $F= F^- + \sum_{i=1}^3 f_i \omega_i$. Then, with respect to $\Theta$, we have
	\begin{equation*}
		\pi_7(F_{\mathbb{A}}) = \sum_{\stackrel{i=1}{(i,j,k) \,\text{cyclic}}}^3 \left( f_i - \frac{1}{2} \left( [\Phi_0,\Phi_i] + [\Phi_j , \Phi_k] \right) \right) \frac{\omega_i - \tau_i}{2}  + \sum_{i=0}^3 J_i^{-1} \left( \sum_{j=0}^3 J_j D \Phi_j \right) \wedge d \theta^i.
	\end{equation*}
Here $J_0=1,$ and $J_i,$ for $i=1,2,3,$ denote the complex structures induced by $\omega_i.$ 
\end{lemma}
\begin{proof}
	First recall that for a $2$-form on $X \times T^4,$ we have $\pi_7(\omega) = \tfrac{1}{4} (\omega + \ast (\omega \wedge \Theta) ).$ Then, we compute $\ast(\omega_i \wedge \Theta)=-2\tau_i + \omega_i$ and $\ast(\tau_i \wedge \Theta)=-2\omega_i + \tau_i,$ hence
	$$\pi_7(\omega_i)= \frac{1}{2}(\omega_i - \tau_i), \qquad \ \pi_7(\tau_i) = \frac{1}{2} (\tau_i - \omega_i).$$
For $\omega \in \Omega^2_-(X) \oplus \Omega^2_-(T^4)$ we have $\pi_7(\omega)=0$. 
	Putting these together, we obtain
	\begin{equation}\label{eq:pi_7_K3_Intermediate_Computation_1}
	\pi_7(F + \sum_{i < j} [\Phi_i , \Phi_j] d \theta_i \wedge d \theta_j) =  \sum_{i=1}^3 \left( f_i - \frac{1}{2} \left( [\Phi_0,\Phi_i] + [\Phi_j , \Phi_k] \right) \right) \frac{\omega_i - \tau_i}{2}.
	\end{equation}
For $\alpha \in \Omega^1(X)$ and $\beta \in \Omega^1(T^4),$ we have
$$\ast (\alpha \wedge \beta \wedge \Theta)= \sum_{i=1}^3 J_i \alpha \wedge I_i\beta$$
where $J_i$, $I_i$ denote the complex structures associated with $\omega_i$ and $\tau_i,$ respectively (which act on $1$-forms by pullback, $(J_i \alpha) (\cdot)=\alpha (J_i \ \cdot)$). The above formula yields
	$$\pi_7(\alpha \wedge \beta) = \frac{1}{4} \left( \alpha \wedge \beta + \sum_{i=1}^3 J_i \alpha \wedge I_i \beta \right) .$$ 
Applying these formulae to the corresponding terms in the curvature of $\mathbb{A},$ we have
	\begin{equation}\label{eq:pi_7_K3_Intermediate_Computation_2}
	\pi_7( \sum_{i=0}^3 D\Phi_i \wedge d \theta_i ) =  \sum_{i=0}^3 J_i^{-1} \left( \sum_{j=0}^3 J_j D \Phi_j \right) \wedge d \theta^i.
	\end{equation}
	Putting (\ref{eq:pi_7_K3_Intermediate_Computation_1}) and (\ref{eq:pi_7_K3_Intermediate_Computation_2}) together yields the formula in the statement.
\end{proof}


\begin{cor}\label{cor:K3}
	Let $\lbrace (A(t), \Phi_0(t), \Phi_1(t), \Phi_2(t), \Phi_3(t) ) \rbrace_{t \in [ 0, T)}$ be a solution of (\ref{eq:Evolution_K3}) on a compact hyperk\"ahler $4$-manifold. Suppose that 
	$$\Vert f_i - \tfrac{1}{2} \left( [\Phi_0,\Phi_i] + [\Phi_j , \Phi_k] \right) \Vert_{L^{\infty}}$$
for $i = 1,2,3$ and $(i,j,k)$ cyclic, and 
	$$ \Vert D\Phi_0 + J_1 D \Phi_1 + J_2 D \Phi_2 + J_3 D \Phi_3 \Vert_{L^{\infty}} $$
	are all uniformly bounded on $[0,T)$. Then the flow can be continued past time $T$.
\end{cor}

\begin{rmk}
	On a hyperk\"ahler $4$-manifold $X$ as above, one may reinterpret the fields $\lbrace \Phi_i \rbrace_{i=1}^4$ as a tuple $(b,c) \in \Omega^2_+(X, \mathfrak{so}(E)) \oplus \Omega^0(X,\mathfrak{so}(E))$ by setting 
	$$b= \frac{1}{\sqrt{2}} \sum_{i=1}^3 \Phi_i \omega_i, \ \text{and} \ c=\frac{1}{\sqrt{2}} \Phi_0.$$ 
	Then, the vanishing of the quantities in Corollary \ref{cor:K3} are precisely the Vafa-Witten equations, as written for example in Section 4.1 of \cite{Haydys2015}. These equations first appeared in \cite{VafaWitten}.
\end{rmk}

\subsection{From $\rSpin(7)$ to $\rSU(3)$} 

Consider a connection $A$ on a bundle $E$ over a Calabi-Yau $3$-fold $X.$ Let $\Phi_1, \Phi_2 \in \Omega^0(X, \mathfrak{so}(E))$, which may be written
$$\Phi=\Phi_1+i\Phi_2 \in \Omega^0(X, \mathfrak{so}(E) \otimes_{\mathbb{R}} \mathbb{C}).$$
Pulling back to $M=X \times T^2$, we obtain a connection $\mathbb{A}= A + \sum_{i=1}^2 \Phi_i \otimes d\theta^i$ on $M,$ with curvature
\begin{equation*}
F_{\mathbb{A}} = F_A + \sum_{i=1}^2 D_A \Phi_i \wedge d \theta^i + [\Phi_1,\Phi_2] d\theta^{12}
\end{equation*}
where $d \theta^{12}=d\theta^1 \wedge d\theta^2$. The Yang-Mills energy of $\mathbb{A}$ yields the following functional 
\begin{equation*}\label{eq:Functional_CY}
\mathcal{E}(A,\Phi) = \frac{1}{2}\int_X |F_A|^2 +  | D_A \Phi_1 |^2 + | D_A \Phi_2 |^2 +  |[\Phi_1,\Phi_2]|^2 . 
\end{equation*}
Computing its negative gradient flow, we obtain the flow equations
\begin{equation}\label{eq:Evolution1_CY}
\begin{split}
\frac{\partial A}{\partial t} & = -D^* F - [\Phi_1 , D \Phi_1 ] - [\Phi_2 , D \Phi_2 ] \\ 
\frac{\partial \Phi_i}{\partial t} & = - \Delta \Phi_i - [\Phi_j , [\Phi_i, \Phi_j]] \qquad \left(i = 1,2, j \neq i\right).
\end{split}
\end{equation}
A similar computation and appeal to the maximum principle, as in the previous case (\ref{eq:Phi_i_Bounded}), shows that $|\Phi_1(t)|$ and $|\Phi_2(t)|$ remain uniformly bounded along the flow (\ref{eq:Evolution1_CY}).

Equip $M=X \times T^2$ with the $\rSpin(7)$-structure given by
\begin{equation}\label{eq:Spin(7)_str_CY}
\Theta = d \theta^{12} \wedge \omega + d \theta^1 \wedge \Omega_1 - d \theta^2 \wedge \Omega_2 + \frac{1}{2} \omega^2
\end{equation}
where $\omega$ is the K\"ahler form and $\Omega=\Omega_1 + i \Omega_2$ the holomorphic volume form on $X$.

\begin{lemma}\label{lem:Spin(7)_to_SU(3)}
	With respect to the $\rSpin(7)$-structure $\Theta$ in (\ref{eq:Spin(7)_str_CY}), we have
	\begin{equation*}
	\begin{split}
	\pi_7 (F_{\mathbb{A}}) & = \left( \Lambda F + [\Phi_1 , \Phi_2] \right) \frac{\omega + d \theta^{12}}{4} \\ \nonumber
	& \quad + \frac{1}{4} \left( (D\Phi_1 - I D\Phi_2) +  \ast_X (F \wedge \Omega_2)  \right) \wedge d \theta^1 \\ \nonumber
	& \quad + \frac{1}{4} \left( (D \Phi_2 + I D\Phi_1) +  \ast_X (F \wedge \Omega_1) \right) \wedge d \theta^2 \\ \nonumber
	& \quad + \frac{1}{4} \Im \left( \ast_X \left( \frac{i}{2} \ast_X (F^{2,0} \wedge \overline{\Omega} ) - D (\Phi_1 - i \Phi_2) \right) \wedge \overline{\Omega} \right),
	\end{split}
	\end{equation*}
	where $\ast_X$ denotes the anti-linear extension to $\Lambda^*_{\mathbb{C}}$ of the Hodge-$\ast$ operator on $X$.
\end{lemma}
\begin{proof}
	It will be convenient to complexify the $2$-forms as $\Lambda^2_{\mathbb{C}} :=\Lambda^2 \otimes_{\mathbb{R}} \mathbb{C}$ and regard $F$ as a (real) section of $\Lambda^2_{\mathbb{C}} \otimes \mathfrak{so}(E)$.
	
	First, one checks that $\ast(d \theta^{12} \wedge \Theta)= \omega$ and $\ast(\omega \wedge \Theta) = 2 \omega + 3 d\theta^{12}$, which can be used to compute
	$$\pi_7(d\theta^{12}) = \frac{1}{4} (\omega + d \theta^{12}) , \qquad \ \pi_7(\omega) = \frac{3}{4} (\omega + d \theta^{12}).$$
	These, together with the fact that $\ast_X(\cdot \wedge \omega)$ has eigenvalues $+1$, $2,$ and $-1$ on the spaces $\Lambda^{2,0} \oplus \Lambda^{0,2},$ $\mathbb{R} \langle \omega \rangle,$ and $\Lambda^{1,1}_0,$ respectively, gives $\Omega^{1,1}_0 (X) \subset \Omega^2_{21} (M),$ and 
	\begin{equation*}
	\begin{split}
	\pi_7(F^{1,1} + [\Phi_1 , \Phi_2] d\theta_1 \wedge d \theta_2) & = \pi_7 \left( \frac{ \Lambda F}{3} \ \omega + [\Phi_1 , \Phi_2] d \theta^{12} \right) \\ \nonumber
	& = \left( \Lambda F + [\Phi_1 , \Phi_2] \right) \frac{\omega + d \theta^{12}}{4} .
	\end{split}
	\end{equation*}

	Next, note that for $\kappa_i \in \Omega^i(X)$, we have
$$\ast(d \theta^{12} \wedge \kappa_4) = \ast_X \kappa_4, \qquad \ast(d\theta^1 \wedge \kappa_5 ) = - d \theta^2 \wedge \ast_X \kappa_5, \qquad \ast( d \theta^2 \wedge \kappa_5) = d \theta^1 \wedge \ast_X \kappa_5.$$
For any $1$-form $\alpha$ on $X,$ this yields
\begin{equation*}
\begin{split}	
	\ast(\alpha \wedge d \theta^1 \wedge \Theta) & = - \ast_X (\alpha \wedge \Omega_2) + I \alpha \wedge d \theta^2 \\
	 \ast(\alpha \wedge d \theta^2 \wedge \Theta ) & = - \ast_X (\alpha \wedge \Omega_1) - I \alpha \wedge d \theta^1.
	 \end{split}
	\end{equation*}
	We obtain
	\begin{equation*}
	\begin{split}
		& \pi_7 (\sum_{i=1}^2 D \Phi_i \wedge d \theta_i) \\
		& \qquad \quad =  \frac{1}{4} \left( (D \Phi_1 - I D \Phi_2) \wedge d \theta^1 + (D \Phi_2 + I D \Phi_1) \wedge d \theta^2 - \ast_X (D \Phi_1 \wedge \Omega_2) - \ast_X (D\Phi_2 \wedge \Omega_1) \right) \\
		& \qquad \quad =  \frac{1}{4} \left(  (D \Phi_1 - I D \Phi_2) \wedge d \theta^1 + (D \Phi_2 + I D \Phi_1) \wedge d \theta^2 - \Im \left( \ast_X ( D (\Phi_1 - i \Phi_2) \wedge \overline{\Omega} \right) \right).
	\end{split}
	\end{equation*}

Furthermore, using the above mentioned fact that for real $\beta \in \Lambda^{2,0} \oplus \Lambda^{0,2},$ we have $\ast_X( \beta \wedge \omega)=\beta,$ together with $\beta= \Re(2\beta^{2,0})=\frac{1}{8} \Im( i \ast_X ( \ast_X (\beta \wedge \overline{\Omega}) \wedge \overline{\Omega}) )$ yields
	\begin{equation*}
	\begin{split}
		\pi_7(F^{2,0}+F^{0,2}) & = \frac{1}{2} (F^{2,0} + F^{0,2}) + \frac{1}{4} \ast_X (F \wedge \Omega_1) \wedge  d\theta^2 + \frac{1}{4} \ast_X (F \wedge \Omega_2 ) \wedge d \theta^1 \\
		& = \frac{1}{8} \Im \left( i \ast_X ( \ast_X ( F^{2,0} \wedge \overline{\Omega}) \wedge \overline{\Omega} ) \right) + \frac{1}{4} \ast_X (F \wedge \Omega_1) \wedge  d\theta^2 + \frac{1}{4} \ast_X (F \wedge \Omega_2 ) \wedge d \theta^1.
		\end{split}
	\end{equation*}
	Putting all these together, we arrive at the formula in the statement.
\end{proof}

The last three lines of $\pi_7(F_{\mathbb{A}})$ in Lemma \ref{lem:Spin(7)_to_SU(3)} only depend on the quantity $\ast_X (F^{2,0} \wedge \overline{\Omega}) + 2i \partial_A \Phi,$ and so we are left with the following conclusion.


\begin{cor}\label{cor:CY}
	Let $\lbrace (A(t), \Phi(t)= \Phi_1(t)-i \Phi_2(t)) \rbrace_{t \in [ 0, T)}$ solve (\ref{eq:Evolution1_CY}) on a compact Calabi-Yau $3$-fold. If the quantities 
	$$\Vert \Lambda F - \tfrac{i}{2} [\Phi, \overline{\Phi}] \Vert_{L^{\infty}} , \qquad  \Vert  \ast_X (F^{2,0} \wedge \overline{\Omega}) + 2i \partial_A \Phi   \Vert_{L^{\infty}} $$
	are both uniformly bounded in $[0,T),$ then the flow can be continued past time $T$.
\end{cor}

\begin{rmk}
	Pairs $(A, \Phi)$ for which the quantities in Corollary \ref{cor:CY} vanish are known both as complex Calabi-Yau monopoles \cite{Oliveira2016} and DT-instantons \cite{Tanaka2014}. 
\end{rmk}

\subsection{From $\rSpin(7)$ to $\rG_2$}

Consider a connection $A$ on a vector bundle $E$ over a $\rG_2$-manifold $(X, \phi),$ and a Higgs field $\Phi \in \Omega^0(X, \mathfrak{so}(E))$. Pulling these back to $M= X \times S^1_{\theta},$ we may define a connection $\mathbb{A}=A+ \Phi \otimes d\theta $ with curvature $F_{\mathbb{A}}= F_A + D\Phi \wedge d\theta.$ 
The Yang-Mills energy of $\mathbb{A}$, with respect to the product metric on $M$, 
is given up to a constant by
\begin{equation*}\label{eq:Yang-Mills-Higgs-G2}
\mathcal{E}(A, \Phi) =  \frac{1}{2}\int_X |F_A|^2 + |D_A \Phi|^2 .
\end{equation*}
Its negative gradient flow is
\begin{equation}\label{eq:Evolution1_G2}
\begin{split}
\frac{\partial A}{\partial t} & = -D^* F - [\Phi , D \Phi ]  \\ 
\frac{\partial \Phi}{\partial t} & = - \Delta \Phi .
\end{split}
\end{equation}
As above, the maximum principle implies a uniform bound on $|\Phi|$ along the flow (\ref{eq:Evolution1_G2}).

Let $\psi=\ast_{\phi} \phi,$ where $\ast_{\phi}$ is the Hodge star operator for the metric on $X$ induced by $\phi,$ and equip $M $ with the $\rSpin(7)$-structure given by
$$\Theta = - \phi \wedge d\theta + \psi.$$
Theorem \ref{thm:mainblowupcrit} implies that the maximal existence time for (\ref{eq:Evolution1_G2}) is characterized by blowup of the 7-component of $F_{\mathbb{A}}$ determined by the $\rSpin(7)$-structure $\Theta$. Having this in mind, one computes
\begin{equation*}
\begin{split}
	\pi_7(F_{\mathbb{A}}) & = \frac{1}{4} (F_{\mathbb{A}} + \ast (F_{\mathbb{A}} \wedge \Phi ) ) \\
	& = \frac{1}{4} \left( F + D\Phi \wedge dt + \ast_{\phi}(F \wedge \phi) - \ast_{\phi} (D\Phi \wedge \psi) + dt \wedge \ast_{\phi}(F \wedge \psi) \right) \\
	& = \frac{1}{4} \left( F + \ast_{\phi}(F \wedge \phi)    - \ast_{\phi} (D\Phi \wedge \psi) \right)  +  \frac{1}{4} dt \wedge \left(  \ast_{\phi}(F \wedge \psi) - D \Phi \right).
\end{split}
\end{equation*}
The above two terms of $\Lambda^2_M \cong  \Lambda^2_X \oplus ( \Lambda^1_X \otimes \Lambda^1_{S^1} )$ can be identified by wedging the second with $\psi,$ and applying $\ast_{\phi}$. This discussion, combined with Theorem \ref{thm:Main_1_Introduction}, proves the following result.

\begin{cor}\label{cor:G2}
	Let $(A(t), \Phi(t))$ be a solution of (\ref{eq:Evolution1_G2}) on a compact $\rG_2$-manifold $(X, \phi),$ for $t \in [ 0, T).$ If 
	$$\Vert D \Phi - \ast_{\phi}(F \wedge \psi) \Vert_{L^{\infty}} $$
	is uniformly bounded on $[0,T)$, then the flow can be continued past time $T$.
\end{cor}

\begin{rmk}
	Pairs $(A, \Phi)$ for which the quantity in Corollary \ref{cor:G2} vanishes, \textit{i.e.} $ D \Phi = \ast_{\phi}(F \wedge \psi),$ are known as $\rG_2$-monopoles. 
\end{rmk}

\subsection{From $\rSpin(7)$ to $\rSU(4)$}

A Calabi-Yau $4$-fold $M,$ with K\"ahler form $\omega$ and holomorphic volume form $\Omega,$ can be equipped with the following torsion-free $\rSpin(7)$-structure:
$$\Theta= \frac{1}{2} \omega \wedge \omega + \Re(\Omega).$$
As before, let $\ast_X$ denote the Hodge-$\ast$ on $X$, which we anti-linearly extend to $\Lambda_{\mathbb{C}}^*$.

\begin{lemma}[Theorem 11.6 in \cite{Salamon2010}]
	With respect to the $\rSpin(7)$-structure determined by the Cayley $4$-form $\Theta$ above, we have
	\begin{equation}\nonumber
	\begin{split}
	\Lambda^2_7 & = \Lambda^{1,1}_{\omega} \oplus \lbrace \beta \in \Re( \Lambda^{2,0}) \ | \ \ast_X (\beta \wedge \Re(\Omega)) = 2 \beta \rbrace  \\ \nonumber
	\Lambda^2_{21} & = \Lambda^{1,1}_0  \oplus \lbrace \beta \in \Re( \Lambda^{2,0}) \ | \ \ast_X (\beta \wedge \Re(\Omega)) = - 2 \beta \rbrace.
	\end{split}
	\end{equation}
\end{lemma}


\begin{cor}\label{cor:SU(4)}
	Let $A(t)$ be a solution of (\ref{YM}) on a compact Calabi-Yau $4$-fold, for $t \in [ 0, T).$ Suppose that 
	$$\Vert \Lambda F \Vert_{L^{\infty}} , \quad  \Vert  \ast_X \left( F^{0,2} \wedge  \frac{\Omega}{4} \right) +  F^{0,2}   \Vert_{L^{\infty}} $$
are both uniformly bounded on $[0,T)$. Then the flow can be continued past time $T$.
\end{cor}

\begin{rmk}
The vanishing of the quantities in Corollary \ref{cor:SU(4)} coincides with the so-called $\mathrm{DT}_4$ equations (called ``$\rSU(4)$-instanton equations'' in \cite{Donaldson1998}). 
\end{rmk}

\end{document}